\documentclass[11pt]{amsart}

\usepackage{amssymb,amsthm}
\usepackage{amsmath}

\newtheorem{thm}{Theorem}
\newtheorem{cor}[thm]{Corollary}
\newtheorem{theorem}{Theorem}[section]
\newtheorem{corollary}[theorem]{Corollary}

\newtheorem{lemma}[theorem]{Lemma}

\def\irr#1{{\rm  Irr}(#1)}

\def\phi{{\varphi}}

\begin{document}


\title[A question of Snyder]{Bounding group orders by large character degrees: A question of Snyder}

\author{Mark L. Lewis}

\address{Department of Mathematical Sciences, Kent State University, Kent, OH 44242}

\email{lewis@math.kent.edu}

\keywords{Character Degrees, solvable groups, Camina pairs, Gagola characters, Suzuki $2$-groups}
\subjclass[2000]{Primary 20C15}

\begin{abstract}
Let $G$ be a nonabelian finite group and let $d$ be an irreducible character degree of $G$.  Then there is a positive integer $e$ so that $|G| = d(d+e)$.  Snyder has shown that if $e > 1$, then $|G|$ is bounded by a function of $e$.  This bound has been improved by Isaacs and by Durfee and Jensen.  In this paper, we will show for groups having a nontrivial, abelian normal subgroup that $|G| \le e^4 - e^3$.  Given that there are a number of solvable groups that meet this bound, it is best possible. Our work makes use of results regarding Camina pairs, Gagola characters, and Suzuki $2$-groups.
\end{abstract}

\maketitle

This paper is dedicated in memory of David Chillag.

\section{Introduction}

Throughout this note, $G$ will be a finite nonabelian group.  We write $d$ for the degree of some nonlinear irreducible character degree of $G$.  Our results are true for any choice of $d$, but they are most interesting when $d$ is the maximal irreducible character degree of $G$, so it is little loss if the reader wishes to make that assumption.

We know that $d$ divides $|G|$, so there is an integer $e$ so that $|G| = d (d+e)$.  Since $d^2 < |G|$, we know that $e$ is a positive integer.  Berkovich has shown that $e = 1$ if and only if $G$ is a $2$-transitive Frobenius group (see Theorem 7 of \cite{Berk}).  It is well known that there are $2$-transitive groups of arbitrarily large order, and so, $d$ may be arbitrarily large.  In this note, we will focus on the case when $e > 1$.

Under the hypothesis that $e > 1$, Snyder has proved that $|G| \le (2e)!$ (see \cite{Snyder}). He also showed that if $e = 2$, then $|G| \le 8$ and if $e = 3$, then $|G| \le 54$, and in both of these cases, there exist examples of these orders; hence, the bounds given are best possible for $e = 2$ and $3$.

Isaacs has shown that $|G| \le B e^6$ for some universal constant $B$ and in many cases that $|G| \le e^6 + e^4$ (see \cite{large}).  Finally, Durfee and Jensen have proved in \cite{DuJe} (without using the classification of nonabelian simple groups) that $|G| \le e^6 - e^4$.  When $G$ is solvable and either $e$ is a prime or $e$ is divisible by at least two distinct primes, they prove that $|G| \le e^4 - e^3$.  Hence, the only time it is possible that $G$ is solvable and $|G| > e^4 - e^3$ is when $e$ is a prime power that is not prime.  Notice that when $e = 2$ and $e = 3$, the expression $e^4 - e^3$ yields the bound found by Snyder.

Isaacs also shows that there exists a (solvable) group $G$ for every prime power $q$ of order $q^3 (q-1)$ where $d = q (q-1)$.  It is easy to compute that $e = q$, so $d = e^2 - e$ and $|G| = e^4 - e^3$.  On the other hand, there are no known groups $G$ where $|G| > e^4 - e^3$, so it seems likely that $|G| \le e^4 - e^3$ is the correct bound.  In this paper, we prove this bound for groups with a nontrivial, abelian subgroup.

\begin{thm} \label{thm2}
Let $G$ be a group with a nontrivial, abelian normal subgroup, and let $d$ and $e$ be defined as above.  If $e > 1$, then $d \le e^2 - e$ and $|G| \le e^4 - e^3$.
\end{thm}

Since all solvable groups have a nontrivial, abelian normal subgroup, we obtain the bound for solvable groups.

\begin{cor}
Let $G$ be a solvable group, and let $d$ and $e$ be defined as
above.  If $e > 1$, then $|G| \le e^4 - e^3$.
\end{cor}

We also are able to make use of the results of Durfee and Jensen to improve the bound for all groups.

\begin{cor} \label{cor1}
Let $G$ be a group with $d$ and $e$ defined as above.  If $e > 1$, then $d < e^2$ and $|G| < e^4 + e^3$.
\end{cor}

At this time, it is still an open question as to whether there exists any groups where $e^4 - e^3 < |G| < e^4 + e^3$.  We do not have any examples of such groups, but at this time, we have not been able to prove that such a group cannot exist.  By Theorem \ref{thm2}, we know that if such a group does exist, then all the normal subgroups of $G$ must be nonabelian.  An important subcase, which we believe is still open is whether the bound $|G| \le e^4 - e^3$ can be proved when $G$ is a simple group.

Like \cite{DuJe}, \cite{large}, and \cite{Snyder}, we consider groups first studied by Gagola.  Gagola studied groups that have an irreducible character that vanish on all but two conjugacy classes.  He proved that if $G$ has a character $\chi \in \irr G$ so that $\chi$ vanishes on all but two conjugacy classes of $G$.  We will say that $\chi$ is a {\it Gagola character}.  Gagola proved that such Gagola characters are unique. Furthermore, he proved that $G$ has a unique minimal normal subgroup $N$.  He proved that $N$ is an elementary abelian $p$-group for some prime $p$.  We say $(G,N)$ is a {\it $p$-Gagola pair} if $G$ has a Gagola character and $N$ is the unique minimal normal subgroup of $G$ and is a $p$-group.  The main result of this paper is the following theorem.

\begin{thm} \label{thm1}
Let $(G,N)$ be a $p$-Gagola pair some prime $p$.  If $P$ is a Sylow $p$-subgroup of $G$, then $d \le e^2 - e$ and $|N|^2 \le |P:N| = |G:N|_p$.
\end{thm}

The proof of this result splits into three cases: when $p$ is odd. when $p = 2$ and $G$ is solvable, and when $p = 2$ and $G$ is nonsolvable.  We have to work the hardest to prove the result when $p = 2$ and $G$ is solvable.  In particular, when $p = 2$ and $G$ is solvable, we need to study in detail Suzuki $2$-groups, and we compute the full automorphism group of a Suzuki $2$-group, which to our knowledge has not been published before.

Durfee and Jensen prove that if $G$ has a nontrivial, abelian normal subgroup, then $G$ has a normal subgroup $N$ so that $(G,N)$ is a $p$-Gagola pair for some prime $p$.  Thus, if there exists a group $G$ with $e^4 - e^3 > |G|$, then $d > e^2 - e$ and all the nontrivial, normal subgroups of $G$ are nonabelian.

We would like to particularly thank David Gluck and Stephen Gagola, Jr. for a number of useful conversations while we were preparing this this paper.

\section{Gagola pairs}

We first use the results in \cite{DuJe} to reduce the initial
question to a question regarding Gagola pairs.  Since $|G| = d
(d+e)$, to prove $|G| \le e^4 - e^3$ it suffices to prove that $d \le e^2 - e$ and to prove $|G| \le e^4 + e^3$, it suffices to prove that $d \le e^2$.  To prove this, we need to introduce some terminology from \cite{DuJe}.  If $\chi, \psi \in \irr G$, then we say $\chi$ {\it dominates} $\psi$ if $\psi \chi = \psi (1) \chi$.  Note that since solvable groups have a nontrivial, abelian normal subgroup, this next lemma applies when $G$ is solvable.

\begin{lemma} \label{gagola 1}
If $G$ has a nontrivial, abelian normal subgroup and if $e > 1$ and $d \ge e^2 - e$ where $d$ and $e$ are defined as above, then $G$ has a Gagola character $\chi$ with $\chi (1) = d$.
\end{lemma}

\begin{proof}
We may apply Theorem 5.2 of \cite{DuJe} to see that $G$ has a
character $\chi \in \irr G$ which dominates all the other
irreducible characters of $G$ and by Lemma 2.1 of \cite{DuJe}, $\chi (1) = d$. By Lemma 4.2 of \cite{DuJe}, we know that $G$ has a minimal normal subgroup $N$ such that $\chi$ vanishes off of $N$ and $G$ acts transitively on $N \setminus \{ 1 \}$.  This implies that $\chi$ vanishes on all but two conjugacy classes of $G$, and so, $\chi$ is a Gagola character.
\end{proof}

When we have a Gagola pair, we can compute $d$ and $e$ in terms of a Sylow subgroup of $G$.

\begin{lemma} \label{gagola 2}
Let $(G,N)$ be a $p$-Gagola pair, and let $P$ be a Sylow
$p$-subgroup of $G$.  If $\chi$ is the associated Gagola character and $d = \chi(1)$, then $e^2 = |P:N|$ and $d = e(|N| - 1)$.
\end{lemma}

\begin{proof}
Let $\chi \in \irr G$ be the Gagola character for $G$.  Let $\lambda \in \irr N$ be a constituent of $\chi_N$.  By Lemma 2.2 of \cite{Gagola}, we may choose $\lambda$ so that $P$ is the stabilizer of $\lambda$, and it is shown in that lemma that $\lambda$ is fully-ramified with respect to $P/N$.  Let $b$ be the integer so that $b^2 = |P:N|$.  We have $\chi (1) = b |G:P|$.  By Corollary 2.3 of \cite{Gagola}, we know that $P$ is a point stabilizer in the action of $G$ on $N \setminus \{ 1 \}$, and from Lemma 2.1 of \cite{Gagola}, we know that $G$ acts transitively on $N \setminus \{ 1 \}$.  By the Fundamental Counting Principle, we have $|G:P| = |N| - 1$, and so, $\chi (1) = b(|N| - 1)$.

We now need to show that $b = e$.  We have
$$
|G| = |G:P| |P| = (|N| - 1)|P:N||N| = d(d+e).
$$
Dividing by $d$, we obtain
$$
d + e = \frac {(|N| - 1)|P:N||N|}{b(|N|-1)} = \frac {|P:N|}b |N| = b|N|,
$$
since $|P:N| = b^2$.  Thus, $e = b|N| - b (|N| - 1) = b$, as
desired.
\end{proof}

We next prove that the first conclusion of Theorem \ref{thm1} is a consequence of the second conclusion.

\begin{corollary} \label{gagola 3}
Let $(G,N)$ be a $p$-Gagola pair, and let $P$ be a Sylow
$p$-subgroup of $G$.  If $|P:N| \ge |N|^2$, then $d \le e^2 - e$.
\end{corollary}

\begin{proof}
By Lemma \ref{gagola 2}, we have $d = e (|N| - 1)$ and $e^2 = |P:N|$.  Assuming $|N|^2 \le |P:N|$ implies that $|N|^2 \le e^2$, and so, $|N| \le e$.  This yields $d \le e(e - 1) = e^2 - e$, as desired.
\end{proof}

We now prove that Theorem \ref{thm2} is a corollary to Theorem \ref{thm1}.

\begin{proof}[Proof of Theorem \ref{thm2}]
By Lemma \ref{gagola 1}, we know since $G$ has a nontrivial, abelian normal subgroup that it has a Gagola character.  Hence, there is a normal $p$-subgroup $N$ so that $(G,N)$ is a $p$-Gagola pair.  By Theorem \ref{thm2}, we know that $|G:N|_p \ge |N|^2$, and applying Corollary \ref{gagola 3}, we obtain $d \le e^2 - e$.  We then obtain $|G| \le (e^2 - e)((e^2 - e) + e) = (e^2 - e)e^2 = e^4 - e^3$.
\end{proof}

With this, we can use the results of Durfee and Jensen to prove Corollary \ref{cor1}.

\begin{proof}[Proof of Corollary \ref{cor1}]
Suppose $d \ge e^2$.  In Corollary 3.4 of \cite{DuJe}, Durfee and Jensen prove that there is a character $\chi \in \irr G$ that dominates all the other characters in $\irr G$.  We then use Lemma 4.1 of \cite{DuJe} to see that $G$ has a normal subgroup $N$ so that $(G,N)$ is a Gagola pair.  We then apply Theorem \ref{thm1} and Corollary \ref{gagola 3} to see that $d \le e^2 - e$.  Since $e > 1$, this contradicts $d \ge e^2$.  Therefore, we conclude that $d < e^2$ and $|G| < e^2 (e^2 + e) = e^4 + e^3$.
\end{proof}

Therefore, for the rest of this paper, we will be concerned with proving the second conclusion of Theorem \ref{thm1} that if $(G,N)$ is a $p$-Gagola pair, then $|G:N|_p \ge |N|^2$.

It turns out that Gagola pairs are examples of a more general
construction that has been studied with some depth.  Let $G$ be a group with a normal subgroup $N$ with $1 < N < G$. We say that $(G,N)$ is a Camina pair if for every $g \in G \setminus N$, the conjugacy class of $g$ contains $gN$. Camina pairs have been studied in a number of different places. There are a number of different conditions that are equivalent to being Camina pairs.  For example, $(G,N)$ is a Camina pair if and only if $|C_G (g)| = |C_{G/N} (gN)|$ for all $g \in G \setminus N$.   In addition $(G,N)$ is a Camina pair if and only if for every every pair of elements $g \in G$ and $x \in N$, there exists an element $y \in G$ so that $[g,y] = x$.  Also, $(G,N)$ is a Camina pair if and only if there is a unique character in $\irr {G \mid \nu}$ for each nonprincipal character $\nu \in \irr N$. Based on this last condition, it follows that if $(G,N)$ is a Gagola pair, then $(G,N)$ is a Camina pair.


%
%

We start with a well-known result regarding Camina pairs where $G$ is a $p$-group and $G/N$ is abelian.

\begin{lemma}\label{abel pair}
If $(G,N)$ is a Camina pair where $G$ is a $p$-group for some prime $p$ and $G/N$ is abelian, then $|G:N| \ge |N|^2$.
\end{lemma}

\begin{proof}
Since $G/N$ is abelian, we have $G' \le N$.  On the other hand, because $(G,N)$ is a Camina pair, we know that every element in $N$ lies in $G'$, so $G' = N$.  Under this hypothesis, it is proved in Theorem 3.2 of \cite{MacD1} that $|N|^2 \le |G:N|$.
\end{proof}

This next result should be compared with Theorem 6.2 of \cite{Gagola}.  In particular, we generalize a result regarding Gagola pairs to Camina pairs.

\begin{lemma} \label{comm}
Let $(G,N)$ be a Camina pair with $N$ a $p$-group for some prime $p$, and let $H$ be a $p'$-subgroup of $G$ such that ${\bf O}_p (G)H$ is normal in $G$.  If $G$ is not a Frobenius group with Frobenius kernel $N$, then $N < [O_p (G), H]$.
\end{lemma}

\begin{proof}
We know that $H$ acts coprimely on ${\bf O}_p (G)$, so
$$
{\bf O}_p (G) = [{\bf O}_p (G),H] C_{{\bf O}_p (G)} (H).
$$
We know from
\cite{ChMc} that $NH$ is a Frobenius group so $N \cap N_G (H) = 1$.  Since ${\bf O}_p (G) H$ is normal in $G$, we may use
the Frattini argument to see that $G = {\bf O}_p (G) N_G (H) = [{\bf
O}_p (G),H] N_G (H)$.  If $N = [{\bf O}_p (G),H]$, then $N$ is
complemented in $G$ by $N_G (H)$.  By Proposition 3.2 of \cite{ChMc}, this
implies that $G$ is a Frobenius group with Frobenius kernel $N$, a
contradiction.
\end{proof}

The following lemma considers the case when we have a ``large'' normal abelian subgroup.

\begin{lemma} \label{abel}
Let $(G,N)$ be a $p$-Gagola pair.  Suppose $M$ is a normal abelian $p$-subgroup of $G$ with $N < M$ and $P$ is a Sylow $p$-subgroup of $G$.  Then $|P:M| \ge |M:N|$.  If in addition, $|N| \le |M:N|$, then $|N|^2 \le |P:N|$.
\end{lemma}

\begin{proof}
Let $\lambda \in \irr N$ with $\lambda \ne 1_N$.  By \cite{Gagola}, we know that $\lambda$ is fully ramified with respect to $P/N$.  It follows that $P/M$ must act transitively on $\irr {M \mid \lambda}$.  We conclude that $|P:M| \ge |\irr {M \mid \lambda}| = |M:N|.$  If $|M:N| \ge |N|$, then $|P:N| \ge |N|^2$, and the result is proved.
\end{proof}

\section{Frobenius complements}

We start with presenting two general lemmas regarding solvable Frobenius complements.  We would not be surprised if these results were known.  A group $G$ is said to be a {\it Z-group} if all of its Sylow subgroups are cyclic.  It is known that if $G$ is a Z-group then $G/G'$ and $G'$ are cyclic groups of coprime order (see Satz IV.2.11 of \cite{hup} or Theorem 5.16 of \cite{isa}).  This first extends to Z-groups that are Frobenius complements, a result that is well-known for Frobenius complements of odd order, and as we will see the proof is essentially identical.  (See Satz V.8.18(b) of \cite{hup} or Theorem 6.19 of \cite{isa}.)

\begin{lemma} \label{Z-group}
Let $G$ be a Frobenius complement that is a Z-group.  If $p$ is a prime that divides $|G|$, then $G$ has a unique subgroup of order $p$.
\end{lemma}

\begin{proof}
We know that $p$ divides only one of $|G'|$ or $|G:G'|$.  If $p$ divides $|G'|$, then since $G'$ is cyclic, we know $G'$ has a unique subgroup of order $p$.  Since $p$ does not divide $|G:G'|$ and $G'$ is normal in $G$, it follows that this is the unique subgroup of $G$ having order $p$.

Suppose that $p$ divides $|G:G'|$.  We know that $G/G'$ is cyclic, so there is a unique subgroup $X/G'$ having order $p$. Observe that $X$ contains all the subgroups of $G$ having order $p$, so it suffices to prove that $X$ has a unique subgroup of order $p$.  Let $P$ be a subgroup of $X$ having order $p$.  Let $Q$ be any subgroup of $G'$ having prime order, say $|Q| = q$. Then $PQ$ is a subgroup of order $pq$, and by either Satz V.8.15 (b) of \cite{hup} or Theorem 6.9 of \cite{isa}, we see that $PQ$ is cyclic.  Hence, $Q$ centralizes $P$.  It follows that every subgroup of $G'$ of prime order centralizes $P$, and thus, $C_{G'} (P)$ contains every subgroup of $G'$ having prime order.  Since $G'$ is abelian and has order coprime to $p$, we may use Fitting's theorem to see that $G' = [G',P] \times C_{G'} (P)$.  Since $C_{G'} (P)$ contains all the subgroups of prime order, we obtain $G' = C_{G'} (P)$. Now, $X = G' \times P$, and this proves the result.
\end{proof}

This next result uses a Theorem due to Zassenhaus.

\begin{theorem} \label{solv Frob}
Let $G$ be a solvable Frobenius complement.  If $p$ is a prime divisor of $|G|$ and $P$ is a subgroup of $G$ of order $p$, then either:
\begin{enumerate}
\item $P$ is normal in $G$; or
\item $p = 3$, $9$ does not divide $|G|$, and  a Sylow $2$-subgroup of $G$ is quaternion of order $8$.
\end{enumerate}
\end{theorem}

\begin{proof}
If $p = 2$, then the result is true by either Satz V.8.18 (a) of \cite{hup} or Theorem 6.3 of \cite{isa}.  Also, if $|G|$ is odd, then we have seen that the result is true.  Therefore, we may assume that $|G|$ is even and $p$ is odd.  By a theorem of Zassenhaus (Theorem 18.2 of \cite{perm}), $G$ has a normal subgroup $N$ so that $N$ is a Z-group and $G/N$ is isomorphic to a subgroup of $S_4$.  Thus, if $p > 3$ or $9$ divides $|G|$, then $p$ divides $|N|$.  Since subgroups of Frobenius complements are Frobenius complements, we may apply Lemma \ref{Z-group} to see that $P$ is characteristic in $N$, and so $P$ is characteristic in $G$.

Thus, we may assume that $p = 3$ and $9$ does not divide $|G|$. Let $H$ be a Hall $2$-complement of $G$ containing $P$.  Thus, $H$ is a Frobenius complement with odd order.  Let $Q$ be a Sylow $2$-subgroup of $G$ so that $QP$ is a subgroup of $G$.  Since $P$ is a subgroup of $H$ with prime order, we know that $P$ is normal in $H$.  If $P$ is central in $PQ$, then $P$ is normal in $G$ since $G = HQ$.  Thus, we may assume that $P$ is not central in $PQ$, and so, the center of $PQ$ is a $2$-group.   By Lemma 18.3 (iii) of \cite{perm}, we know that either $P$ is normal in $PQ$ and hence $G$, or the Fitting subgroup of $PQ$ is isomorphic to the quaternions.  This implies that $Q$ is quaternion of order $8$.
\end{proof}

\section{Semi-linear groups}

This first result of this section is suggested by Theorem 2.4 of \cite{Gagola} and Theorem B of \cite{Kuisch}.  We would like to be able to replace the hypothesis that $(G,N)$ is a Gagola pair with the hypothesis that $(G,N)$ is a Camina pair with $N$ an elementary abelian $p$-group (that is minimal normal in $G$).

If $V$ is an elementary abelian $p$-group, then we define $\Gamma (V)$ to be the semi-linear group as defined in Chapter 2 of \cite{MaWo}.  We can view $V$ as the additive group of some finite field $F$.  In particular, $\Gamma = \Gamma (V)$ is the semi-direct product of $S$ acting on $\Gamma_o (V)$ where $\Gamma_o (V)$ is isomorphic to the multiplicative group of $F$ and $S$ is isomorphic to the Galois group of $F$ over $Z_p$.  If $|V| = p^n$, then $|\Gamma_o (V)| = p^n - 1$ and $|S| = n$. It is not difficult to show that $C_{\Gamma} (\Gamma_o (V)) = \Gamma_o (V)$.

If $p$ is a prime, then recall that the group $G$ is {\it $p$-closed} if $G$ has a normal Sylow $p$-subgroup.

\begin{lemma} \label{aff}
Let $(G,N)$ be a $p$-Gagola pair where $G$ is solvable.  If $G$ is not $p$-closed, then either
\begin{enumerate}
\item $G/{\bf O}_p (G) \le \Gamma (N)$ and $|N| = p^{ap}$ for some positive integer $a$, and $G/{\bf O}_p (G)$ has a normal $p$-complement, or
\item $p = 3$, $|N| = 9$, and $G/{\bf O}_3 (G)$ is isomorphic to ${\rm SL}_2 (3)$.
\end{enumerate}
\end{lemma}

\begin{proof}
By Corollary 2.3 of \cite{Gagola}, we know that $G/{\bf O}_p (G)$ acts on $N$ and that $C_{G/{\bf O}_p (G)} (x)$ is a Sylow $p$-subgroup of $G/{\bf O}_p (G)$ for every $x \in N \setminus \{ 1 \}$.  This satisfies the hypotheses of Lemma 1 of \cite{bound}.  The conclusion of Lemma 1 of \cite{bound} is that either $|N| = 9$ and $G/{\bf O}_3 (G)$ is isomorphic to either ${\rm SL}_2 (3)$ or ${\rm GL}_2 (3)$, or $G/{\bf O}_p (G) \le \Gamma (N)$.  If the first conclusion holds, then we have our conclusion (2) since $|G:{\bf O}_p (G)| = |N| - 1 = 48$.  Thus, we may assume that $G/{\bf O}_p (G) \le \Gamma (N)$.  Using the notation from \cite{MaWo}, we know that $\Gamma (N)$ has a normal cyclic subgroup $\Gamma_o (N)$ whose order is $|N| - 1$.  Because $N$ is a $p$-group, we have $|N| = p^n$ for some positive integer $n$.  Since $p$ divides $|G:{\bf O}_p (G)|$ and $p$ does not divide $|\Gamma_o (N)| = p^n - 1$, we conclude that $p$ divides $|\Gamma (N):\Gamma_o (N)| = n$, and thus, $n = ap$ for some positive integer $a$ as desired.   Since $\Gamma (N)/\Gamma_o (N)$ is cyclic, we see that $G/{\bf O}_p (G)$ has a normal $p$-complement.
\end{proof}

This next result is a number theoretic condition.  It is probably known.  It is related to the number theoretic results proved in Lemma 2.4 of \cite{nor}.

\begin{lemma}\label{nor cond}
Let $p$ be a prime and $n$ a positive integer.  Suppose that $q$ is a prime so that $n = q^a m$ where $a$ is a positive integer and $m$ is not divisible by $q$.  If $q^2$ divides $p^m - 1$, then $(p^n - 1)_q = q^a (p^m - 1)_q$.
\end{lemma}

\begin{proof}
Observe that
$$
p^n - 1 = \prod_{i=1}^a \left( \frac {p^{m q^i} - 1}{p^{m q^{i-1}} - 1} \right) (p^m - 1).
$$
Thus, it suffices to prove that $\displaystyle \left( \frac {p^{m q^i} - 1}{p^{m q^{i-1}} - 1} \right)_q = q$ for all $i$. By Lemma 2.4 of \cite{nor}, we know that $q$ divides $({p^{m q^i} - 1})/({p^{m q^{i-1}} - 1})$ and that ${\rm gcd} (p^{m q^{i-1}} -1,({p^{m q^i} - 1})/({p^{m q^{i-1}} - 1}))$ divides $q$.  Since $q^2$ divides $p^m - 1$ which divides $p^{m q^{i-1}} - 1$, it follows that $\displaystyle \left( \frac {p^{m q^i} - 1}{p^{m q^{i-1}} - 1} \right)_q = q$.
\end{proof}

We now apply this to obtain a result about subgroups of semi-linear groups.

\begin{lemma} \label{prime pow}
Let $V$ be an elementary abelian group of order $2^n$.  Let $H$ be a subgroup of $\Gamma (V)$ of order $2^n - 1$ that acts transitively on $V \setminus \{ 1 \}$.  If $d$ is a nontrivial prime power that divides $n$, then $H \cap \Gamma_o (V)$ contains an element of order $2^d - 1$.
\end{lemma}

\begin{proof}
Since $\Gamma_o (V)$ is cyclic, it suffices to show that $H \cap \Gamma_o (V)$ contains an element of order $(2^d - 1)_r$ for every prime $r$ that divides $2^d - 1$.  Fix a prime divisor $r$ of $2^d - 1$.  Suppose $(2^d - 1)_r = r$.  We know that $H$ has a normal subgroup of order $r$.  Thus, this subgroup will centralize $H \cap \Gamma_o (V)$.  We know that the centralizer of $\Gamma_o (V)$ in $\Gamma (V)$ is $\Gamma_o (V)$, so $H \cap \Gamma_o (V)$ will contain the subgroup of order $r$.  Thus, we may assume that $(2^d - 1)_r \ge r^2$.

Let $p$ be the prime so that $d$ is a power of $p$.  We know that $2^d \cong 1 ~({\rm mod}~r)$, so the order of $2$ modulo $r$ must be a power of $p$.  This implies that $p$ divides $r - 1$, and so, $r$ does not divide $d$.  If we write $n = r^a m$ where $a$ is a positive integer and $r$ does not divide $m$, then we conclude that $2^d - 1$ will divide $2^m - 1$.  In particular, $r^2$ divides $2^m - 1$.

Let $R$ be a Sylow $r$-subgroup of $\Gamma (V)$ so that $R \cap H$ is a Sylow $r$-subgroup of $H$.  We know that $|R| = (2^n - 1)_r n_r$.  Observe that $R \cap \Gamma_o (V)$ is a Sylow $r$-subgroup of $\Gamma_o (V)$, so $|R \cap \Gamma_o (V)| = (2^n - 1)_r$.  It follows that $|R:R \cap \Gamma_o (V)| = n_r = r^a$. Since $H$ is a Frobenius complement, we know that $R \cap H$ is cyclic.  Let $x$ be a generator for $R \cap H$.  Since $R \cap \Gamma_o (V)$ is normal in $R$, we see that $x^{r^a} \in R \cap \Gamma_o (V) \le H \cap \Gamma_o (V)$.  The order of $x^{r^a}$ will be $(2^n - 1)_r/r^a$, and by Lemma \ref{nor cond}, this equals $(2^m - 1)_r$.  Since this is divisible by $(2^d - 1)_r$, this yields the desired conclusion.
\end{proof}

\section{$p$ odd} \label{thm1sec}

We now prove Theorem \ref{thm1} when $p$ is odd.

We begin with a result that applies both when $p$ is odd and when $p = 2$.

\begin{lemma} \label{|N|p^2}
Let $(G,N)$ be a $p$-Gagola pair.  If $|N| \le p^2$, then $|G:N|_p \ge |N|^2$.
\end{lemma}

\begin{proof}
In \cite{Gagola}, it is proved that $N < {\bf O}_p (G)$, $N \le Z ({\bf O}_p (G))$, and $|G:N|_p$ is a square.  Consider an element $x \in {\bf O}_p (G)$, and let $D (x) = \{ g \in G | [x,g] \in N \}$.  We know that $D(x)/N = C_{G/N} (xN)$, so $|D(x):N| = |C_G (x)|$.  This implies that $|D(x):C_G(x)| = |N|$.  Notice that $\langle N,x \rangle \le C_G (x)$.  Since $xN$ will have nontrivial $p$-power order. This implies that $|G:N|_p > |N|$.

If $|N| = p$, then $|G:N|_p \ge p^2 = |N|^2$.  If $|N| = p^2$, then $|G:N|_p \ge p^3$.  Since we know that $|G:N|_p$ is square, this implies that $|G:N| \ge p^4 = |N|^2$.  This proves the lemma.
\end{proof}

Let $p$ be a prime and $a$ a positive integer.  We say that $q$ is a {\it Zsigmondy} prime of $p^a - 1$ if $q$ divides $p^a - 1$ and $q$ does not divide $p^b - 1$ for every positive integers $b < a$.  The Zsigmondy prime theorem states that a Zsigmondy primes exists for primes $p$ and positive integers $a$ except when $p$ is a Mersenne prime (i.e., $p + 1$ is a power of $2$) and $a = 2$ and when $p = 2$ and $n = 6$.  (See
Theorem IX.8.3 of \cite{HBII}.)

Let $p$ be a prime greater than $3$.  We know that $p^2$ is congruent to $1$ modulo $3$, so $3$ divides $p^2 - 1$.  This implies the following fact.  If $p$ is prime, $n$ is a positive integer, and $3$ is a Zsigmondy prime divisor of $p^n - 1$, then $n \le 2$.

We begin with a lemma about $p$-groups of nilpotence class $3$.

\begin{lemma} \label{class3}
Let $p$ be an odd prime, and suppose $P$ is a $p$-group of nilpotence class $3$.  Suppose that $[P',P] \le Z \le Z(P) \cap P'$.  If $P$ has an automorphism of order $2$ that inverts all the elements in $P/P'$ and in $Z$ and centralizes the elements in $P'/Z$, then for every element $x \in P \setminus P'$, there is an element $y \in xP'$ so that $C_{P/Z} (y Z) = C_P (y)P'/Z$.
\end{lemma}

\begin{proof}
Let $\sigma$ be the given automorphism of order $2$.  For each element $g \in P$, define $D (g)/Z = C_{P/Z} (gZ)$.  Since $P'/Z$ is central in $P/Z$, we see that if $h \in gP'$, then $D(h) = D (g)$.

Suppose $r \in P \setminus P'$.  Define $R = \langle P',r \rangle$.  Since $P'/Z$ is central, $R/Z$ is abelian.  We know that $\sigma$ inverts $rP'$, so $\sigma$ acts on $R$.  Notice that $C_{R/Z} (\sigma) = P'/Z$, so by Fitting's lemma, $R/Z = P'/Z \times [R,\sigma]/Z$.  Thus, the coset $rP'$ contains an element $s \in [R,\sigma]$.  Notice that $\langle s, Z \rangle$ is abelian, and $\sigma$ acts Frobeniusly on both $\langle s, Z \rangle/Z$ and $Z$, so $\sigma$ acts Frobeniusly on $\langle s, Z \rangle$.  We conclude that $\sigma$ inverts $s$.

Now, consider $x \in P \setminus P'$.  By the previous paragraph, we can find an element $y \in xP'$ so that $y^\sigma = y^{-1}$.  Observe that $C_P (y) \le D (y)$ and $P' \le D (y)$, so $C_P (y) P' \le D (y)$.

Consider $d \in D(y)$.  We need to show that $d \in C_P (y) P'$.  If $d \in P'$, then this holds, so we may assume $d \not\in P'$.  By the previous paragraph, we can find $c \in dP'$ so that $c^\sigma = c^{-1}$.  We know that $c \in D (y)$, so $[c,y] \in Z$.  It follows that $[c,y]^\sigma = [c^\sigma,y^\sigma] = [c^{-1},y^{-1}] = [c,y]$.  Since $\sigma$ inverts all the elements of $Z$, this implies that $[c,y] = 1$, and so, $c \in C_P (y)$.  We now have $dP' = cP' \le C_P (y) P'$, and we conclude that $d \in C_P (y) P'$.  This proves that $D (y) = C_P (y)P'$, and this proves the result.
\end{proof}

This next theorem proves Theorem \ref{thm1} under the hypothesis that $p$ is an odd prime.  For nonsolvable groups, we apply a result that is found in \cite{Gagola} which relied on the classification of finite simple groups.  Thus, the nonsolvable case of this result also relies on the classification of finite simple groups.  On the other hand, if one assumes solvability, then the classification is not needed.

\begin{theorem} If $(G,N)$ is a $p$-Gagola pair where $p$ is odd, then $|N|^2 \le |G:N|_p$.
\end{theorem}

\begin{proof}
By Lemma \ref{|N|p^2}, we may assume that $|N| = p^a$ where $a \ge 3$.  We claim that it suffices to find a normal $p$-subgroup $M$ and a subgroup $Y$ of order $2$ so that $[M',M] \le N \le M'$, $|M:M'| \ge |N|$, and $Y$ acts Frobeniusly on $M/M'$.  (We know that $Y$ must act Frobeniusly on $N$.)  If $M' = N$, then since $C_N (Y) = 1$, we conclude that $C_M (Y) = 1$.  It follows that $Y$ acts Frobeniusly on $M$.  We conclude that $M$ is abelian which is a contradiction.  Thus, we may assume that $N < M'$.  Since $[M',M] \le N$ and $N$ is central in $M$ (see page 367 of \cite{Gagola}), this implies that $M$ has nilpotence class $3$.

We know that $C_M (Y) \le M'$ since $Y$ acts Frobeniusly on $M/M'$.  Since $M'/N$ is central in $M/N$, we see that $C_M (Y) N$ is normal in $M$.  Notice that $Y$ will act Frobeniusly on $M/(C_M (Y) N)$, so $M/(C_M (Y) N)$ is abelian, and thus, $M' = C_M (Y) N$. It follows that $Y$ centralizes $M'/N$.  Hence, we may apply Lemma \ref{class3}.  By that result, we can find an element $x \in M \setminus M'$ so that $C_{M/N} (xN) = C_M (x) M'/N$.  Let $D (x)/N = C_{G/N} (xN)$, and this implies that $D (x) \cap M = C_M (x) M'$.

We now compute
$$
|M D(x):C_G (x) \cap M'| = |M D(x):M| |M:M'| |M': C_G (x) \cap M'|.
$$
We know that $|MD (x):M| = |D(x): D(x) \cap M| = |D(x): C_M (x)M'|$ and $|M':C_G (x) \cap M'| = | C_M (x) M':C_M (x)|$.  We have
$$
|D(x): C_M (x)M'| | C_M (x) M':C_M (x)| = |D (x):C_M (x)|.
$$
Since $|D (x): C_M (x)| = |D (x): C_G (x) \cap M|$ is divisible by $|D(x):C_G (x)| = |N|$ and $|M:M'| \ge |N|$, this yields the desired conclusion.  We now work to find such subgroups $M$ and $Y$.

Suppose first that $G$ is solvable.  Since $a \ge 3$, we know that $|N| - 1 = p^a - 1$ has a Zsigmondy prime divisor $q$.  As we noted earlier, $q \ne 3$.  Let $H$ be a Hall $p$-complement of $G$.  We know from Lemma 4.3 in \cite{ChMc} that $NH$ is a $2$-transitive Frobenius group.  Since $a \ge 3$, we may apply Huppert's theorem (Proposition 19.10 of \cite{perm} or Theorem 6.9 of \cite{MaWo}) to see that $H$ is isomorphic to a subgroup of $\Gamma (N)$, and so, $H$ has a normal subgroup $K$ of order $q$.  We claim that ${\bf O}_p (G) K$ is normal in $G$.  If $G$ is $p$-closed, then this is immediate.  If $G$ is not $p$-closed, then the existence of a Zsigmondy prime implies that we are in case 2 of Lemma \ref{aff}. In particular, $G/{\bf O}_p (G) \le \Gamma (N)$, and thus, ${\bf O}_p (G)K$ is normal in $G$.  Hence, we may use Lemma \ref{comm} to see that $N < [{\bf O}_p (G), K] = M$.  Notice that $K$ acts Frobeniusly on $M/M'$.  If $M' = 1$, then $K$ acts Frobeniusly on $M/N$, and since $|K| = q$ is a Zsigmondy prime divisor of $p^a - 1$, we have $|M:N| \ge |N|$.  And by Lemma \ref{abel}, we have the result.  Thus, we may assume that $M' > 1$, and since $N$ is the unique minimal normal subgroup of $G$, this implies that $N \le M'$.  This implies that ${\bf O}_p (G) = M C_{{\bf O}_p (G)} (K)$, and by the Frattini argument, $G = {\bf O}_p (G) N_G (K) = M N_G (K)$. If $M/M'$ is not irreducible under the action of $K$, then $|M:M'| \ge p^{2n}$, and the result follows.  Thus, we may assume that $M/M'$ is irreducible under the action of $K$.


Since $p$ is odd, $2$ divides $|H|$.  Let $Y \le H$ be a subgroup of order $2$.  We know that $Y$ is normal in $H$ by Theorem \ref{solv Frob}.  In particular, $K$ centralizes $Y$.  Hence, $N_G (K)$ normalizes $C_M (Y)$.  Since $M/M'$ is abelian, $M$ normalizes $C_M (Y) M'$, and so, $C_M (Y) M'$ is normal in $G$.  If $M = C_M (Y) M' = C_M (Y) \Phi (M)$, then Frattini's theorem implies $M = C_M (Y)$, and this is a contradiction since $Y$ acts Frobeniusly on $N$.  Thus, $C_M (Y) M' < M$.  Since $M/M'$ is irreducible under the action of $K$, we have $C_M (Y) \le M'$.  We deduce that $Y$ acts Frobeniusly on $M/M'$.


If $K$ acts nontrivially on $M'/N$, then we will have $|M':N| \ge |N|$, and the result will follow since $|M:N| \ge |N|^2$.  Thus, we may assume that $K$ centralizes $M'/N$.  It follows that $[M',K] \le N$.  We then have $[M',K,M] \le [N,M] = 1$ since $N \le Z (M)$, and $[M,M',K] \le [M',K] = N$.  By the three subgroups lemma, this implies that $[K,M,M'] \le N$, and since $M = [K,M]$, we conclude that $[M,M'] \le N$.  We now have the desired subgroups $M$ and $Y$, and the result follows in this case.

Finally, we may assume $G$ is not solvable and $|N| = p^a$ where $a > 2$.  In this case, we may apply Theorem 5.6 of \cite{Gagola} to see that $G/{\bf O}_p (G)$ either has ${\rm SL} (2,q)$ where $q$ is a power of $p$ as normal subgroup whose quotient is a $p$-group, has ${\rm SL} (2,5)$ as a normal subgroup of index $2$ that is a nonsplit extension, or is isomorphic to ${\rm SL} (2,13)$.  In all of these cases, $G/{\bf O}_p (G)$ has a unique minimal normal subgroup $L/{\bf O}_p (G)$ which is of order $2$.  Taking $Y$ to be a subgroup of order $2$ in $L$, we have $L = {\bf O}_p (G) Y$ is normal in $G$.

Again, we may apply Lemma \ref{comm} to find $M$ so that $N < M = [{\bf O}_p (G), Y]$.  It is not difficult to see that $Y$ acts Frobeniusly on $M/M'$.  Let $C/N = C_{G/N}{M/N}$.  Since ${\bf O}_p (G)/N$ is a $p$-group, we know that $Z({\bf O}_p (G)/N) \cap M/N > 1$, and since $M/N$ is a chief factor for $G$, it follows that $M/N \le Z ({\bf O}_p (G)/N)$.  This implies that ${\bf O}_p (G) \le C$.  We see that $C \cap L= {\bf O}_p (G)$.  Since $L/{\bf O}_p (G)$ is the unique minimal normal subgroup of $G/{\bf O}_p (G)$, we conclude that $C = {\bf O}_p (G)$.

Let $K$ be a subgroup of $G$ of order $q$.  Since $K$ is not contained in $C$, we see that $K$ acts nontrivially on $M/N$.  If $M' = 1$, then it acts Frobeniusly on $M$, and this implies that $|M:N| \ge |N|$, we obtain the result by Lemma \ref{abel}.  Thus, we may assume that $M' > 1$.  By the uniqueness of $N$, this implies $M' \ge N$.  Since $K$ acts nontrivially on $M/N$, it acts nontrivially on $M/M'$.  It follows that $|M:M'| \ge |N|$.  If $K$ acts nontrivially on $M'/N$, then $|M':N| \ge |N|$, and the result holds.  Thus, we may assume that $K$ centralizes $M'/N$.  This implies that the centralizer of $M'/N$ in $G/{\bf O}_p (G)$ is nontrivial.  By the uniqueness of $L/{\bf O}_p (G)$, this implies that $L$ centralizes $M'/N$, and so $Y$ centralizes $M'/N$.  We have $[M',Y,M] \le [N,M] \le N$.  Also, $[M,M',Y] \le [M',Y] \le N$.  By the Three Subgroups Lemma (see Lemma 4.9 of \cite{isa}), $[Y,M,M'] \le N$, and since $[M,Y] = M$, we obtain $[M,M'] \le N$.  We now have subgroups $M$ and $Y$ as desired.    This completes the proof.
\end{proof}






\section{$p=2$ preliminary results}

We now left with the case that $N$ is a $2$-group.  We break up the case when $p = 2$ into two subcases: $G$ is solvable and $G$ is not solvable.  In this section, we include some results that are common to both cases.

\begin{lemma} \label{five}
Let $K$ have normal subgroup $N < M$ so that $N$ is an elementary abelian $2$-group, $M/N$ is cyclic of odd order, $|K:M| = 2$, $C_{N}(M/N) = 1$, and $K/N$ is a dihedral group.  Then $K \setminus M$ contains an involution.
\end{lemma}

\begin{proof}
Let $C$ be a Hall $2$-complement of $M$.  Since $C$ is a Hall subgroup of $M$, we can apply the Frattini argument to see that $K = M N_G(C) = NC N_G (C) = N N_G (C)$.  This implies that $2$ divides $|N_G (C)|$.  Let $A = N_G (C) \cap N$.  Then $[A,C] \le C$ and $[A,C] \le [N,C] \le N$, so $[A,C] \le C \cap N = 1$.  This implies that $[A,C] = 1$.  We deduce that $[A,M] = [A,NC] \le [A,N][A,C] = [A,N] \le N$, and we conclude that $A \le C_{N}(M/N) = 1$.  Thus, we have $N_G (C) \cap N = 1$.  Let $T$ be a Sylow $2$-subgroup of $N_G (C)$.  It follows that $T > 1$ and $T \cap N = 1$.  Thus, $T$ contains an involution which is not in $N$.  Since $N$ is the Sylow $2$-subgroup of $M$, this implies that $T$ contains an involution not in $M$.
\end{proof}

The following fact is essentially the heart of the proof when $p = 2$ of the proof of Theorem 5.1 of \cite{large}.

\begin{lemma} \label{six}
Let $(G,N)$ be a $2$-Gagola pair.  Then every involution in $G/N$ lies in ${\bf O}_2 (G)/N$.
\end{lemma}

\begin{proof}
Let $t \in G$ so that $tN$ is an involution.  If $t \in {\bf O}_2 (G)$, then the result holds; so assume $t \not\in {\bf O}_2 (G)$.  In particular, we know that $t$ does not centralize $N$ by \cite{Gagola}.  By Theorem 2.13 of \cite{isa}, $tN$ must invert some nontrivial element $cN \in G/N$ of odd prime order.  Let $M = \langle N, c \rangle$ and let $K = \langle M, t \rangle$.  Let $T = \langle N, t \rangle$, and observe that $T$ is a Sylow $2$-subgroup of $K$.  Notice that $T$ is not abelian, so $T$ must contain an element of order $4$.  Since $N$ is elementary abelian, this element must lie in $T \setminus N$.  Since $T:N| = 2$, we see that $T \setminus N = Nt$.  Since $(G,N)$ is a Gagola pair, we know that all the elements in $Nt$ are conjugate, and so, they all have order $4$.  On the other hand, we may now apply Lemma \ref{five} to see that $K \setminus N$ must contain an involution $s$.  This yields a contradiction since $s$ must be conjugate to some element of $T \setminus N$, and we have seen that this set contains no involutions.
\end{proof}

We begin with the following observation which appears in my paper \cite{pCam}.

\begin{lemma} \label{seven}
Let $(G,N)$ be a $2$-Gagola pair.  If ${\bf O}_2 (G)/N$ does not have exponent $2$, then $|G:N| \ge |N|_2$.  In particular, if ${\bf O}_2 (G)/N$ is not elementary abelian, then $|G:N|_2 \ge |N|^2$.
\end{lemma}

\begin{proof}
Consider $x \in {\bf O}_2 (G)$ so that $x^2 \not\in N$.  By \cite{Gagola}, we have $N \le C_G (x)$.  Let $D (x)/N = C_{G/N} (x)$.  We know that $C_G (x) \le D (x)$ and $|D (x):N| = |C_G (x)|$, so $|D(x):C_G (x)| = |N|$.  Suppose that there exists $x \in {\bf O}_2 (G)$   Consider $y \in D (x)$.  Then $[x,y] \in N$.  We observe that $[x^2,y] = [x,y]^x [x,y]$, and since $N$ is central in ${\bf O}_2 (G)$, we have $[x,y]^x = [x,y]$.  Thus, $[x^2,y] = [x,y]^2$, and since $N$ is elementary abelian, we conclude that $[x,y]^2 = 1$.  It follows that $y \in C_G (x^2)$.  Thus, $D (x) \le C_G (x^2)$.  Since $x^2 \in {\bf O}_2 (G)$, we deduce that $|N|^2 = |D(x^2):C_G(x^2)||D(x):C_G (x)|$ divides $|G:N|$.  This proves the result when ${\bf O}_2 (G)/N$ does not have exponent $2$.
\end{proof}

\section{$p=2$ and $G$ nonsolvable}

We continue to work on the case when $p = 2$.  In this section, we consider the subcase where $G$ is nonsolvable.  We begin with some number theoretic results.  We would not be surprised if these results were known.

\begin{lemma} \label{one}
Let $a$ be a nonnegative integer.  Then $2^{3^a} \equiv -1 ~({\rm mod}~3^{a+1})$ and $2^{3^a} \not\equiv -1 ~({\rm mod}~3^{a+2})$.
\end{lemma}

\begin{proof}
We work by induction on $a$.  Notice that $2^{3^0} = 2^1 = 2 \equiv -1 ~({\rm mod}~3^{0+1})$ and $2 \not\equiv -1 ~({\rm mod}~3^{1+1})$.  This proves the base case.

We now prove the inductive step.  Suppose for some nonnegative integer $a$, we have $2^{3^a} \equiv -1 ~({\rm mod}~3^{a+1})$ and $2^{3^a} \not\equiv -1 ~({\rm mod}~3^{a+2})$.  This implies that $2^{3^a} = -1 + b 3^{a+1}$ where $b$ is an integer that is not divisible by $3$.  Cubing, we obtain $2^{3^{a+1}} = (2^{3^a})^3 = (-1 + b 3^{a+1})^3$.  Using the binomial theorem, we obtain $(-1 + b 3^{a+1})^3 = -1 + 3 b 3^{a+1} - 3 b^2 3^{2(a+1)} + b^3 3^{3(a+1)} = -1 + 3^{a + 2} c$ where $c = b (1 - b 3^{a+1} + b^2 3^{2a+1})$.  It follows that $2^{3^{a+1}} \equiv -1 ~({\rm mod}~3^{a+2})$.  Since $3$ does not divide $b$, it follows that $3$ does not divide $c$, and so, $2^{3^{a+1}} \not\equiv -1 ~({\rm mod}~3^{a+3})$.  This proves the inductive step, and hence, the lemma is proved.
\end{proof}

We make use of the following corollary.

\begin{corollary}\label{onea}
Suppose that $n_3 = 3^a$.  Then $2^n \equiv -1 ~({\rm mod}~3^{a+1})$ if $n$ is odd and $2^n \equiv 1 ~({\rm mod}~3^{a+1})$ if $n$ is even.  In particular, $\left((2^n - 1)(2^n + 1)\right)_3 = 3^{a+1}$.
\end{corollary}

\begin{proof}
Write $b$ so that $n = 3^a b$, and note that $3$ does not divide $b$.  Using Lemma \ref{one}, we have $2^n = (2^{3^a})^b \equiv (-1)^b ~({\rm mod}~3^{a+1})$.  This gives the first conclusion.  Observe that $2^n - 1$ and $2^n + 1$ are relatively prime.  Thus, $3$ only divides one of these.  Suppose now that $n$ is odd.  Then $2^n + 1 = (2^{3^a} + 1) (\sum_{i=0}^{b-1} (-1)^i (2^{3^a})^i)$.  Let $c = \sum_{i=0}^{b-1} (-1)^i (2^{3^a})^i$.  It follows that $c \equiv \sum_{i=0}^{b-1} (-1)^i(-1)^i \equiv \sum_{i=0}^{b-1} 1 \equiv b ~({\rm mod}~3)$ (this uses the fact from Lemma \ref{one} that $3^a \equiv -1 ~({\rm mod}~3)$.  It follows that $3$ does not divide $c$, so $(2^n+1)_3 = (2^{3^a} + 1)_3$, and applying Lemma \ref{one} again, we have that $3^{a+1}$ divides $2^n + 1$ and $3^{a+2}$ does not divide $2^n + 1$, so $(2^n + 1)_3 = 3^{a+1}$.  This gives the result $n$ is odd.  If $n$ is even, then $n = 2m$ where $m_3 = n_3$.  Working by induction on $n$, we have $(2^n - 1)_3 = (2^m - 1)(2^m + 1)_3 = 3^{a+1}$, and this proves the result when $n$ is even.
\end{proof}

We now get an application to the order of $|{\rm SL}_2 (2^n)|$.

\begin{corollary} \label{two}
Suppose that $n_3 = 3^a$.  Then $|{\rm SL}_2 (2^n)|_3 = 3^{a+1}$.
\end{corollary}

\begin{proof}
We know that $|{\rm SL}_2 (2^n)| = (2^n - 1)2^n (2^n + 1)$.  
This now follows immediately from Corollary \ref{onea}.
\end{proof}

\begin{lemma}\label{three}
Suppose that $V$ is the natural module for $S = {\rm SL}_2 (2^n)$ over $F = {\rm GF} (2^n)$, and let $\phi$ be the character afforded by $V$.  If $P$ is a Sylow $3$-subgroup of $S$, then $\phi_P = \lambda + \lambda^{-1}$ where $\lambda$ is a faithful character of $P$ (over the algebraic closure of $F$).
\end{lemma}

\begin{proof}
First, we note that $P$ is cyclic, so $P = \langle x \rangle$ for some element $x$.  Let ${\mathcal X}$ be the representation afforded by $V$, so its trace is $\phi$.  We know by Brauer-Nesbitt that $\phi$ is absolutely irreducible.  Let $b$ be the integer so $3^b = |S|_3$.  Take $E$ to be an extension of $F$ that contains a $3^b$th root of unity $\eta$.  Since $\phi$ is absolutely irreducible, we may work with $V^E$ and ${\mathcal X}$.  Conjugating if necessary, we may assume that ${\mathcal X}^E (x) = \left[ \begin{array} {ll}
                       \eta & 0 \\
                       0 & \eta^{-1} \\
                       \end{array}\right] .$
Define $\lambda$ to the character of $P$ over $E$ defined by $\lambda (x) = \eta$.  Since $\eta$ and $P$ have the same order, we see that $\lambda$ is faithful.  It is not difficult to see that $\phi_P = \lambda+\lambda^{-1}$.
\end{proof}

The next result contains one of the main pieces of our argument.  In particular, we get a restriction on the dimensions of irreducible modules of ${\rm SL}_2 (2^n)$ over $Z_2$.

\begin{lemma} \label{four}
If $V$ is an irreducible module for ${\rm SL}_2 (2^n)$ over $Z_2$ and if ${\rm dim}_{Z_2} (V) < 4n$, then ${\rm dim}_{Z_2} (V)$ is either $1$, $2n$, or $8n/3$, and $8n/3$ occurs only when $3$ divides $n$.  Furthermore, if ${\rm dim}_{Z_2} (V) = 8n/3$ and $P$ is a Sylow $3$-subgroup of ${\rm SL}_2 (2^n)$, then $C_V (P) = 0$.
\end{lemma}

\begin{proof}
Let $F$ be the field with $2^n$ elements.  Let $C$ be the Galois group for $F/Z_2$, so $C$ is cyclic of order $n$.  Let $W$ be the natural module for $G = {\rm SL}_2 (2^n)$.  We know, viewed as a module for $G$ over $F$, that $W$ has dimension $2$.  (Viewed as a module over $Z_2$, it follows that $W$ has dimension $2n$.)  Let $\phi$ be the $F$-character of $G$ afforded by $W$.  By the Brauer-Nesbitt theorem \cite{Nesb}, we know that every absolutely irreducible character for $G$ has the form $\phi^S = \prod_{\sigma \in S} \phi^\sigma$ where $S$ is a subset of $C$.  Notice that this implies that $F$ will be the splitting field for $G$.

We now write $\psi$ for the character afforded by $V^F$.  By Theorem 9.21 of \cite{text}, we know that $\psi$ is a sum of Galois conjugates of an irreducible $F$-character whose multiplicity is $1$.  Fix the subset $S$ of $C$ so that $\phi^S$ is an irreducible $F$-constituent of $\psi$.  We see that $\psi$ will be the sum of the Galois conjugates of $\phi^S$.

Notice that if $S$ is empty, then $\psi = \psi^S = 1$, and we conclude that the dimension of $V$ is $1$.  Thus, we may assume that $S$ is not empty.  Consider an element $c \in C$.  Then $(\phi^S)^c = \phi^{Sc}$. Hence, $(\phi^S)^c = \phi^{Sc}$ if and only if $S = Sc$, and so, the number of Galois conjugates of $\phi^S$ is equal to the number of sets of the form $Sc$.  Let $T = \{ c \in C \mid S = Sc \}$.  By the Fundamental Counting Principle (see Theorem 1.4 of \cite{isa}), the number of sets of the form $Sc$ equals $|C:T|$.  Fix an element $s \in S$.  If $t \in T$, then $st \in S$.  This implies that $t \in s^{-1} S$, and so, $T \subseteq s^{-1}S$.  We conclude that $|T| \le |S|$.

We see that ${\rm dim}_{Z_2} (V) = {\rm dim}_F (V^F) = {\rm deg} (\psi) = |C:T| ({\rm deg} \phi)^{|S|} = n2^{|S|}/|T| \ge n2^{|S|}/|S|$.  If $|S| = 1$, then $|T| = 1$, and we have ${\rm dim}_{Z_2} (V) = n2^1/1 = 2n$.  If $|S| = 2$, then either $|T| = 1$ or $|T| = 2$.  If $|T| = 1$, then ${\rm dim}_{Z_2} (V) = n2^2/1 = 4n$, and if $|T| = 2$, then ${\rm dim}_{Z_2} (V) = n2^2/2 = 2n$.  If $|S| = 3$, then $|T|$ is either $1$, $2$, or $3$.  If $|T| \le 2$, then ${\rm dim}_{Z_2} (V) = n2^3/|T| \ge n8/2 = 4n$.  If $|T| = 3$, then ${\rm dim}_{Z_2} (V) = n2^3/3 = 8n/3$.  Notice that $|T|$ divides $|C| = n$, so $3$ divides $n$.  Finally, if $|S| \ge 4$, then ${\rm dim}_{Z_2} (V) \ge n 2^|S|/|S|$.  Consider the function $f (x) = 2^x/x$.  Observe that $f'(x) = 2^x(x {\rm ln} (2) - 1)/x^2$, and so, $f' (x) > 0$ when $x > 1/{\rm ln} (2) \cong 1.44$.  This implies that $f (x)$ is increasing when $x \ge 4$.  In particular, $2^{|S|}/|S| \ge 2^4/4 = 16/4 = 4$, and conclude that ${\rm dim}_{Z_2} (V) \ge 4n$.  This proves the first conclusion.

We now focus on the case where ${\rm dim}_{Z_2} (V) = 8n/3$.  Notice that we must have $|T| = |S| = 3$.  Let $P$ be a Sylow $3$-subgroup of $G$.  We need to show that $C_V (P) = 0$. It suffices to show that $C_{V^F} (P) = 0$.  Since $\psi$ is the character afforded by $V^F$, this will follow if we can prove that $1_P$ is not a constituent of $\psi_P$.

Notice that $\psi$ contains all the characters of the form $\phi^{Sc}$, so we may assume that $1 \in S$.  It is difficult to see that this implies that $S = T$.  Let $c$ be the Frobenius automorphism of $F$, so $c$ is a generator of $C$.  Write $n = 3m$ where $m$ is an integer.  It is not difficult to see that $S = \{ 1, c^m, c^{2m} \}$, and so, $\psi = \sum_{i=0}^{m-1} (\phi \phi^{c^m} \phi^{c^{2m}})^{c^i}$.

We now consider $\psi_P$.  We wish to show that the principal character of $P$ is not a constituent of $\psi_P$.  We know that if $x \in F$, then $x^c = x^2$.  It follows that $\phi^{c^i} = \phi^{2^i}$, and so, $(\phi \phi^{c^m} \phi^{c^{2m}})^{c^i} = (\phi \phi^{c^m} \phi^{c^{2m}})^{2^i}$.  Thus, if $1_P$ is not a constituent of $\phi_P {\phi_P}^{c^m} {\phi_P}^{c^{2m}}$, then $1_P$ will not be a constituent of $(\phi \phi^{c^m} \phi^{c^{2m}})^{c^i}$.  Thus, to show that $1_P$ is not a constituent of $\psi_P$, it suffices to show that $1_P$ is not a constituent of $\phi_P {\phi_P}^{c^m} {\phi_P}^{c^{2m}}$.

By Lemma \ref{three}, we know that $\phi_P = \lambda + \lambda^{-1}$ where $\lambda$ is a faithful character of $P$ over an extension of $F$.  It follows that
$$
\phi_P {\phi_P}^{c^m} {\phi_P}^{c^{2m}} = (\lambda + \lambda^{-1})(\lambda^{2^m} + \lambda^{-2^m})(\lambda^{2^{2m}} + \lambda^{-2^{2m}}).
$$
Thus, all of the irreducible constituents of the character $\phi_P {\phi_P}^{c^m} {\phi_P}^{c^{2m}}$ have the form $\lambda^{a_1 + a_2 2^m + a_3 2^{2m}}$ where $a_1, a_2, a_3 \in \{ \pm 1 \}$.

Let $n_3 = 3^a$, and note that this implies that $m_3 = 3^{a-1}$.  We now apply Corollary \ref{onea} to see that $a_1 +a_2 2^m + a_3 2^{2m} \equiv a_1 + a_2 (-1) + a_3 (-1)^2 ~({\rm mod}~3^a)$ when $m$ is odd, and $a_1 +a_2 2^m + a_3 2^{2m} \equiv a_1 + a_2  + a_3 ~({\rm mod}~3^a)$ when $m$ is even.  This gives possible values of $\pm 1, \pm 3 ({\rm mod} 3^a)$.  Since $\lambda$ has order $3^{a+1}$, it follows that $1_P$ will not be occur when $a > 1$.  Suppose that $a = 1$.  We see that either $m$ is congruent to either $1$ or $2$ modulo $3$.  If $m$ is congruent to $1$ modulo $3$, then $2^m \equiv 2 ~({\rm mod}~9)$ and $2^{2m} \equiv 4 ~({\rm mod}~9)$.  This yields $a_1 +a_2 2^m + a_3 2^{2m} \equiv a_1 + 2a_2 + 4a_3 ~({\rm mod} ~9)$, and we obtain the possible values $\pm 3, \pm 5, \pm 7$ modulo $9$.  If $m$ is congruent to $2$ modulo $3$, then $2^m \equiv 4 ~({\rm mod}~9)$ and $2^{2m} \equiv 2~({\rm mod}~9)$.  Again we obtain the values $\pm 3, \pm 5, \pm 7$ modulo $9$.  As before, we conclude that $1_P$ does not occur, and this proves the result.
\end{proof}

Finally, we come to the main result of this section.  This proves Theorem \ref{thm2} when $p = 2$ and $G$ is nonsolvable.

\begin{theorem}
Let $(G,N)$ be a $2$-Gagola pair where $G$ is not solvable.  Then $|G:N|_2 \ge |N|^2$.
\end{theorem}

\begin{proof} by way of contradiction, we assume that a counterexample $G$ exists.  We see the result holds if ${\bf O}_2 (G)/N$ is not elementary abelian by Lemma \ref{seven}.  Thus, ${\bf O}_2 (G)/N$ is elementary abelian.  By Lemma \ref{six}, we know that every involution in $G/N$ lies in ${\bf O}_2 (G)/N$.  By Theorem 5.5 \cite{Gagola}, we know that there a $2$-power $q$ so that $|N| = q^2$ and a subgroup $S$ of $G$ so that ${\bf O}_2 (G) \le S \le G$ where $S$ is normal in $G$, $|G:S|$ is a power of $2$, $S/{\bf O}_2 (G) \cong {\rm SL}_2 (q)$, $N$ is the natural module for $S/{\bf O}_2 (G) \cong {\rm SL}_2 (q)$ and $G/{\bf O}_2 (G) \le {\rm Aut} (S/{\bf O}_2 (G))$.

We now prove a series of claims.

Claim 1: $Z ({\bf O}_2 (G)) = N$.  Let $Z = Z ({\bf O})_2 (G)$. We know that $N \le Z$ by \cite{Gagola}.  Suppose $N < Z$, and pick $z \in Z \setminus N$.  It follows that ${\bf O}_2 (G) \le C_G (z)$.  Let $D(z) = \{ g \in G | [g,z] \in N \}$.  We know that $|D(z):C_G (z)| = |N| = q^2$ and that $|D (z):C_G (z)|$ divides $|G:{\bf O}_2 (G)|$.  We know that $|G:{\bf O}_2 (G)|$ divides $|{\rm Aut} ({\rm SL}_2 (q)|$.  Thus, $|G:{\bf O}_2 (G)|_2$ divides $q {\rm log}_2 (q)_2 < q^2$.  We now have a contradiction.  Thus, $Z = N$, and the claim is proved.

Since we know that $N < {\bf O}_2 (G)$, Claim 1 implies that ${\bf O}_2 (G)$ is nonabelian.  We now show that often we may assume that $q > 4$.

Claim 2: If $|{\bf O}_2 (G):N| > q^2$ and $q = 4$, then the theorem is proved.  We know from \cite{Gagola} that $|G:N|_2$ is a square.  We see that $|S:{\bf O}_2 (G)|_2 |{\bf O}_2 (G):N|$ divides $|G:N|_2$.  This implies that $2q^2 4 = 2^7$ divides $|G:N|_2$.  Since $|G:N|_2$ is a square, $2^8$ divides $|G:N|_2$.  We conclude that $|G:N|_2 \ge 2^8 = 4^4 = q^4 = |N|^2$.

We now consider subgroups $A$ with $N \le A \le {\bf O}_2 (G)$ and $[A,S] \le N$.

Claim 3: If there exists a subgroup $A$ with $N \le A \le {\bf O}_2 (G)$ and $[A,S] \le N$, then $A$ is abelian, and in particular, $A < {\bf O}_2 (G)$.  Observe that $N$ is irreducible under the action of $S$.  Thus, either $A' = 1$ or $A' = N$.  Suppose $A' = N$, and thus, $N < A$.  Then it follows that $A$ is nonabelian.  Let $p$ be an odd prime divisor of $|S|$, and write $P$ for a Sylow $p$-subgroup of $S$.  We see that $[A,P] \le [A,S] \le N$, so $P$ centralizes $A/N$.  Hence, $A = C_A (P) N = C_A (P) A' \le C_A (P) \Phi (A) \le A$.  Thus, $A= C_A (P) \Phi (A)$, and this implies $A = C_A (P)$.  We conclude that $P$ centralizes $A$ which is a contradiction since $P$ act Frobeniusly on $1 < N \le A$.  Therefore, we must have $A' = 1$, and the claim is proved.

Claim 4: There does not exist a subgroup $A$ with $N < A < {\bf O}_2 (G)$ and $[A,S] \le N$.  By way of contradiction, assume such a subgroup $A$ exits.  Suppose $a \in A \setminus N$. Observe that $[a,S] \le N$, so $S/N \le C_{S/N} (aN)$.  Since $(G,N)$ is a Gagola pair, we know that $|C_G (a)| = |C_{G/N} (aN)|$.  Because $|G:S|$ is a power of $2$, this implies that $C_G (a)$ contains a full Sylow $p$-subgroup for every odd prime $p$ that divides $|G|$.  Let $p$ be an odd prime divisor of $|G|$, and let $P$ be a Sylow $p$-subgroup of $C_G (a)$.  Observe that $P \le S$, so $P$ is a Sylow $p$-subgroup of $C_S (a)$.  It follows that $C_S (a)$ contains a full Sylow subgroup for every odd prime.  Since $S/{\bf O}_2 (G) \cong {\rm SL}_2 (q)$ has no proper subgroups with this property, we conclude that $S = {\bf O}_2 (G) C_S (a)$.

Observe that $[{\bf O}_2 (G), a] \le N$.  Since $N$ is central in ${\bf O}_2 (G)$, this implies that $[{\bf O}_2 (G),a]$ is normal in ${\bf O}_2 (G)$.  Obviously, $C_S (a)$ normalizes $[{\bf O}_2 (G), a]$, so $[{\bf O}_2 (G), a]$ is normal in $S$.  Since $a \not\in N = Z ({\bf O}_2 (G))$, we see that $[{\bf O}_2 (G), a] > 1$.  Since $N$ is irreducible under the action of $S$, we conclude that $[{\bf O}_2 (G), a] = N$.  Hence, the map ${\bf O}_2 (G) \rightarrow N$ defined by $x \mapsto [x,a]$ is onto.  Since $N$ is central in ${\bf O}_2 (G)$, this map is a homomorphism and its kernel is $C_{{\bf O}_2 (G)} (a)$.  By the isomorphism theorems, we conclude that $|{\bf O}_2 (G): C_{{\bf O}_2 (G)} (a)| = |N| = q^2$.

We note that $C_{{\bf O}_2 (G)} (a)$ is normal in $C_S (a)$.  Notice that $N \le C_{{\bf O}_2 (G)} (a)$, and we have ${\bf O}_2 (G)/N$ is elementary abelian, so $C_{{\bf O}_2 (G)} (a)$ is normal in ${\bf O}_2 (G)$.  It follows that $C_{{\bf O}_2 (G)} (a)$ is normal in $S = {\bf O}_2 (G) C_S (a)$.

We can view $C_{{\bf O}_2 (G)} (a)/N$ as a module for $S/{\bf O}_2 (G)$.  If this module has any nonprincipal constituents, then we may apply Lemma \ref{four} to see that $|C_{{\bf O}_2 (G)} (a)/N| \ge q^2$.  This implies that $|{\bf O}_2 (G):N| = |{\bf O}_2 (G): C_{{\bf O}_2 (G)} (a)| |C_{{\bf O}_2 (G)} (a):N| \ge q^2 q^2 = q^4 = |N|^2$, and the theorem is proved.  Thus, we may assume that $[C_{{\bf O}_2 (G)} (a), S] \le N$.

We have $C_S (a)/ C_{{\bf O}_2 (G)} (a) \cong S/{\bf O}_2 (G) \cong {\rm SL}_2 (q)$, and we have just shown that $C_{{\bf O}_2 (G)} (a)/N \le Z (C_S (a)/N)$.  Let $C = C_S (a)$.  It is not difficult to see that $C' \cap {\bf O}_2 (G) = C' \cap C_{{\bf O}_2 (G)} (a)$.  If $N < C' \cap C_{{\bf O}_2 (G)} (a)$, then ${\rm SL}_2 (q)$ must have a nontrivial Schur multiplier, and so, $q = 4$.  Let $p$ be an odd prime divisor of $|S|$, and let $P$ be a Sylow $p$-subgroup of $S$.  Since $P$ centralizes $A/N$, and $P$ does not centralize ${\bf O}_2 (G)/N$, we conclude that $P$ does not centralize ${\bf O}_2 (G)/A$.  Thus, $S$ does centralize ${\bf O}_2 (G)/A$, so we may apply Lemma \ref{four} to see that $|{\bf O}_2 (G):A| \ge q^2$, and hence, $|{\bf O}_2 (G):N| > q^2$.  With $q = 4$, we have seen in Claim 2 that the Theorem is proved.  Thus, we may assume that $C' \cap C_{{\bf O}_2 (G)} (a) = N$.

Now, we have that $C'/N \cong S/{\bf O}_2 (G)$.  Obviously, $C'/N$ contains an involution.  But any such involution will not lie in ${\bf O}_2 (G)/N$, and this contradicts our earlier observation.  This contradiction proves the claim.

Claim 5: $[{\bf O}_2 (G),S]/N$ is irreducible under the action of $S$.  Suppose $N \le M < [{\bf O}_2 (G),S]$ so that $[{\bf O}_2 (G),S]/M$ is chief for $G$.  By Lemma \ref{four}, we have $|[{\bf O}_2 (G),S]:M| \ge q^2$.  If $S$ acts nontrivially on $M/N$, then $|M:N| \ge q^2$ by Lemma \ref{four}, and $|{\bf O}_2 (G):N| \ge q^4$ as desired.  Thus, we may assume that $[M,S] \le N$.  By Claim 4, this implies that $M = N$.  This proves the claim.

We now set $M = [{\bf O}_2 (G),S]$.

Claim 6: $S' \cap {\bf O}_2 (G) = M$.  Observe that ${\bf O}_2 (G)/M \le Z (S/M)$.  If $S' \cap {\bf O}_2 (G) > M$, then ${\rm SL}_2 (q)$ has a nontrivial Schur multiplier, so $q = 4$. We have $|{\bf O}_2 (G):N| > q^2$, so we are done in this case. This proves the claim.

Claim 7: $|M:N| = q^2$.  If $|M:N| \ge q^4$, then we have the result.  Thus, we may assume that $|M:N| < q^4$.  By Lemma \ref{four}, we know that if $|M:N| > q^2$, then $|M:N| = 2^{8n/3}$ where $q = 2^n$ and $3$ divides $n$.  We also know that if $P$ is a Sylow $3$-subgroup of $S'$, then $C_{M/N} (P) = 1$.  Let $D/M = N_{S'/M} (PM/M)$.  From Dickson's classification of the subgroups of ${\rm SL}_2 (q)$, we see that $D/M$ is a dihedral group.  Thus, we may apply Lemma \ref{five} to see that $D/N$ contains an involution not in $M/N$, but this yields an involution in $G/N$ that is outside ${\bf O}_2 (G)/N$ which is a contradiction.  This implies that $|M:N| = q^2$.

Claim 8: Final contradiction.  If $M$ is abelian, then the result holds by Lemma \ref{abel}.  Thus, $M$ is nonabelian.  This implies that $M' = N$.  Since $M/M'$ is a chief factor, we obtain $N = Z (M)$.  We know the exponent of $M$ must be $4$, so there exists $x \in M \setminus N$ so that $x^2 \ne 1$.  Notice that every element in $xN$ must have order $4$, since they are all conjugate to $x$.  We know that $S$ acts transitively on $N \setminus \{ 1 \}$.  This implies that the $S$-orbit of $x^2$ has size $q^2 - 1$.  It follows that the $S$-orbit of $xN$ must size at least $q^2 - 1$.  Since $|M/N \setminus \{ N \}| = q^2 - 1$, we conclude that $S$ acts transitively on $M/N$.  In particular, $C{S'/M}{xN} = C_{S'/M} (x^2) = 1$.  If $p$ is an odd prime divisor of $|S|$ and $P$ is a Sylow $p$-subgroup of $S'$, then $C_{M/N}(P) = 1$.  Let $D/M = N_{S'/M} (PM/M)$.  From Dickson's classification of the subgroups of ${\rm SL}_2 (q)$, we see that $D/M$ is a dihedral group.  Thus, we may apply Lemma \ref{five} to see that $D/N$ contains an involution not in $M/N$, but this yields an involution in $G/N$ that is outside ${\bf O}_2 (G)/N$ which is a contradiction.  This the final contradiction and the theorem is proved.
\end{proof}

\section{$p=2$ and $G$ solvable: part I}

We now start to handle the case where $p = 2$ and $G$ is solvable.  Recall that a non-abelian $2$-group $H$ is a Suzuki $2$-group if $H$ has more than one involution and ${\rm Aut} (H)$ has a solvable subgroup that acts transitively on the involutions of $H$.  The Suzuki $2$-groups of type A are denoted by $A (n,\Theta)$ where $n$ is a positive integer and $\Theta$ is a nontrivial odd-order automorphism of the field of order $2^n$.  The order of $A (n,\Theta)$ is $2^{2n}$.  We will give a more complete description of these groups in the next section.  At this time, we have not been able to dispense with the solvable hypothesis.

\begin{theorem} \label{2pairclass}
If $(G,N)$ is a $2$-Gagola pair where $G$ is solvable,
then either:
\begin{enumerate}
\item $|N|^2 \le |G:N|_2$ or
\item
    \begin{enumerate}
    \item $G$ is not $2$-closed.
    \item Writing $|N| = 2^n$ for an integer $n$, then $n$ is divisible by both $2$ and some odd prime.
    \item $G = M N_G (K)$ with $M \cap N_G (K) = 1$.
    \item ${\bf O}_ 2 (G) = M C$ where $M \cap C = 1$,
        $M = [{\bf O}_2 (G),K]$ and $C = C_{{\bf O}_2 (G)} (K)$ where $K$ is a subgroup of order $q$ where if possible $q$ is a Zsigmondy prime divisor of $|N| - 1$ and otherwise $q = 3$ when $|N| = 2^6$.
    \item $M$ is a Suzuki $2$-group of type A with $M' = Z(M) = N$.
    \item $C_G (M) = N$, so $N_G (K)$ is isomorphic to a subgroup of ${\rm Aut} (M)$.
    \item If $H$ is Hall $2$-complement of $G$, then $MH$ is a Frobenius group.
    \item If $x \in N_G (K)$, then $C_G (x) = C_M (x) C_{N_G (K)} (x)$ and $C_M (x) \le N$.  In particular, if $x \in C$, then $C_M (x) = N$.
    \end{enumerate}
\end{enumerate}
\end{theorem}

\begin{proof}
%

By Lemma \ref{seven}, we have the first conclusion if ${\bf O}_2 (G)/N$ is not elementary abelian.  Thus, we may assume ${\bf O}_2 (G)/N$ is elementary abelian.If $G$ is $2$-closed, then we may use \cite{ChMc} to see that $({\bf O}_2 (G),N)$ is a Camina pair.  Since ${\bf O}_2 (G)/N$ is abelian, we may use Lemma \ref{abel pair} to see that $|{\bf O}_2 (G):N| \ge |N|^2$, and this gives the desired conclusion.  Thus, $G$ is not $2$-closed.

Let $H$ be a Hall $2$-complement of $G$.  If $|N| - 1$ has a Zsigmondy prime divisor $q$, then let $K = K^*$ be the subgroup of $H$ of order $q$.  Otherwise, $|N| = 2^6$ and we take $K$ to be the subgroup of $H$ of order $3$ and we take $K^*$ to be the subgroup of $H$ order $9$.  By Lemma \ref{aff}, we know that $G/{\bf O}_2 (G) \le \Gamma (N)$, and so, ${\bf O}_2 (G)K$ is normal in $G$. Applying Lemma \ref{comm}, we can find $M$ so that $N < M \le [{\bf O}_2 (G),K]$ with $M$ normal in $G$ and $M/N$ a chief factor of $G$.  Notice that $K$ will act Frobeniusly on $M/N$ since $M/N$ is abelian, so $K^*$ acts Frobeniusly on $M/N$.  Also, $M/N$ will be a direct sum of irreducible $K^*$ modules. It follows that $|N| \le |M:N|$.  If $M/N$ is not irreducible under the action of $K^*$, then $|N|^2 \le |M:N|$ and we are done.  Thus, $M/N$ is irreducible under the action of $K^*$ and so, $|M:N| = |N|$.  Also, if $M < [{\bf O}_2 (G),K]$, then $|N| \le |[{\bf O}_2 (G),K]:M|$ and $|N|^2 \le |P:N|$ where $P$ is a Sylow $2$-subgroup of $G$.  We have $M = [{\bf O}_2 (G),K]$.  We obtain ${\bf O}_2 (G) = M C_{{\bf O}_2 (G)} (K)$ and by the Frattini argument, $G = {\bf O}_2 (G) N_G (K) =  M N_G (K)$.  Let $C = C_{{\bf O}_2 (G)} (K)$, so ${\bf O}_2 (G) = MC$.

If $M$ is abelian, then we are done by Lemma \ref{abel}.  We may assume that $M$ is not abelian.  This implies that $M' = \Phi (M) = Z(M) = N$.

We have $C_G (M) \cap M = Z(M) = N$.  Let $B = C_G (M)$.  We will prove that $B = N$.  Note that $B$ is a normal subgroup of $G$ and $B \cap H = 1$, so $B$ is a $2$-group, and we obtain $B \le {\bf O}_2 (G)$.  We have $B = [B,K] C_{B} (K)$.  Now, $N = [N,K] \le [B,K] \le [{\bf O}_2 (G),K] \cap B \le M \cap B = N$.  Also, $C_{B} (K) \cap N \le C \cap N = 1$.  Now, $M$ centralizes $B$, so $N$ will centralize $C_{B} (K)$, and $N_G (K)$ normalizes both $B$ and $K$, so $N_G (K)$ normalizes $C_B (K)$.  It follows that $G = M N_G (K)$ normalizes $C_B (K)$.  Since $N$ is the unique minimal normal subgroup of $G$, this implies that $C_B (K) = 1$ and $B = [B,K] = N$.  In particular, $G/N$ is isomorphic to a subgroup of the automorphism group of $M$.


Let $X$ be any subgroup of $H$ of prime order.  By Lemma \ref{solv Frob}, we know that $X$ is normal in $H$, so $K$ centralizes $X$.  In particular, $N_G (K)$ normalizes $C_M (X)$.  Now, $M$ normalizes $C_M (X) N$, and thus, $C_M (X) N$ is normal in $G$.  If $M = C_M (X) N = C_M (X) \Phi (M)$, then $M = C_M (X)$ which is a contradiction since $X$ acts Frobeniusly on $N$.  Hence, $C_M (X) \le N$, and since $C_N (X) = 1$, we have $C_M (X) = 1$.  In particular, $H$ acts Frobeniusly on $M$.  This implies that $C \cap M = 1$.  Notice that $N_G (K) \cap {\bf O}_2 (G) = C$, so this implies $N_G (K) \cap M = 1$.

Since $|M:N| = |N| = |H| + 1$, we see that $H$ acts transitively on $M/N$.  Since $M$ is not abelian, there exists $m \in M \setminus N$ of order $4$.  It is easy to see that every element in the coset $mN$ will have order $4$ since $N = Z (M)$ and $N$ is elementary abelian.  Because $H$ acts transitively on $M/N$, this implies that every element in $M \setminus N$ has order $4$.  It follows that $N$ contains all the involutions in $M$.  Since $H$ acts transitively on $N \setminus \{ 1 \}$, $H$ acts transitively on the involutions of $M$.  We deduce that that $M$ is a Suzuki $2$-group.  By Theorem VIII.7.9 of \cite{HBII}, $M$ is isomorphic to $A (m,\Theta)$ as defined in Section VIII.6 of \cite{HBII} where $|N| = 2^n$ and $\Theta$ is an automorphism of the field of order $2^n$.  By Theorem VIII.6.9 of \cite{HBII}, we see that $\Theta$ must have odd order not equal to $1$.  On the other hand, by Lemma \ref{aff}, we know that $2$ divides $n$.

Because $N$ is central in $M$ and elementary abelian, the map $bN \mapsto b^2$ is a well-defined function from $M/N$ to $N$. Notice that if $h \in H$, then $bN^h = b^hN$ and $(b^2)^h = (b^h)^2$, so the action of $H$ commutes with this function.  Since $H$ is acting transitively, it follows that the function must be a bijection.

Suppose $x \in N_G (K)$.  Obviously, we have $C_M (x) C_{N_G (K)} (x) \le C_G (x)$.  Consider an element $g \in C_G (x)$.  We can write $g = mc$ where $m \in M$ and $c \in N_G (K)$.  We know that $1 = [g,x] = [mc,x] = [m,x]^c[c,x]$, and this implies that $[m,x]^c = [c,x]^{-1}$.  We know that $[m,x]^c \in M$ and $[c,x]^{-1} \in N_G (K)$.  This implies that $[m,x]^c = [c,x]^{-1}$ lie in $N_G (K) \cap M = 1$, and so $[m,x] = [c,x] = 1$, and thus, $m \in C_M (x)$ and $c \in C_{N_G (K)} (x)$.  We conclude that $C_G (x) = C_M (x) C_{N_G (K)} (x)$.

A similar argument shows that
$$
C_{G/N} (xN) = C_{M/N} (xN) C_{N_G(K) N/N} (xN).
$$
Making the observation that $C_{N_G (K) N/N} (xN) \cong C_{N_G (K)} (x)$, we have
$$
|C_{G/N} (xN)| = |C_{M/N} (xN)| |C_{N_G (K)} (x)|.
$$
Recalling that $|C_G (x)| = |C_{G/N} (xN)|$, we  conclude that $|C_{M/N} (xN)| = |C_M (x)|$.  Notice that the action of $x$ must commute with the map $bN \mapsto b^2$ from $M/N$ to $N$, and so, we must have $|C_{M/N} (xN)| = |C_N (x)|$.  Since this implies that $|C_N (x)| = |C_M (x)|$, we must have $C_M (x) \le N$.  If $x \in C$, then $C_M (x) = N$.
\end{proof}

\section{Automorphisms of Suzuki groups}

In order to finish the proof of Theorem \ref{thm1} we need to understand conclusion 2 of Theorem \ref{2pairclass}, and to do this we compute the automorphism group of a Suzuki $2$-group of type A and how it acts on the group.

There are two papers in Russian which we found in translation that address the automorphism groups of Suzuki $2$-groups of type A.  The first, \cite{autos}, computes the Sylow $2$-subgroup of the quotient of the
automorphism group by its centralizer of the center.  In particular,
it show that the Sylow $2$-subgroup is cyclic, and it uses that to
show that the automorphism group is solvable.

The second one,
\cite{p-alg}, finds three subgroups of the automorphism group.  We
will discuss the three subgroups later.  They then prove results for
an analog of the Suzuki $2$-groups for odd primes.  They mention
that they consider odd primes so that they do not need the ``subtle
considerations needed in the case of $p = 2$.''


We now recall the structure of a Suzuki $2$-group of type A. Since
$M$ is isomorphic to $A (n,\Theta)$, we can write the elements of $M$
as $(a,b)$ where $a,b \in F$ and $F = {\rm GF} (2^n)$. The
multiplication is defined by $(a,c) \cdot (b,d) = (a + b, c + d + b
\Theta (a))$.  (Recall that $\Theta$ is a nontrivial automorphism of
$F$ having odd order.)  The elements of $N$ are $(0,b)$ as $b$ runs
through $F$. Observe that $(a,b) = (a,0)(0,b) \in (a,0) N$. Since
$N$ is central ${\bf O}_2 (G)$, we have $C_{{\bf O}_2 (G)} ((a,0)) =
C_{{\bf O}_2 (G)} ((a,b))$.  Also, we see that $(a,b)N = (a,0)N$, so
$C_{G/N} ((a,0)N) = C_{G/N} ((a,b)N)$.

Let $M$ be a Suzuki $2$-group of Type A and use the notation of the
previous paragraph.  In \cite{p-alg}, they define three subgroups of
${\rm Aut} (M)$.  They are:

\begin{enumerate}
\item $A_1 = \{ \phi = \phi_\psi \mid (a,b)^\phi = (a, \psi (a) + b
\}$ where $\psi$ runs over the linear transformations of $F$
regarded as a vector space over $Z_2$.

\item $A_2 = \{ \phi = \phi_x \mid (a,b)^\phi = (xa,x \Theta (x) b) \}$
where $x$ runs over the nonzero elements of $F$.

\item $A_3 = \{ \phi = \phi_\tau \mid (a,b)^\phi = (a^\tau,b^\tau)
\}$ where $\tau$ runs over the Galois automorphisms of $F$ with
respect to $Z_p$.
\end{enumerate}

It is not difficult to show that $|A_1| = 2^{n^2}$.  (Each of the
$n$ basis elements of $F$ can be sent to any element of $F$.)  Also,
we see that $|A_2| = 2^n - 1$ and $|A_3| = n$.  Suppose $\psi$ is a
linear transformation of $F$, $x \in F \setminus \{ 0 \}$, and
$\tau$ is a Galois automorphism of $F$.  Consider $(a,b) \in M$.
Then
$$
(\phi_\psi)^{\phi_x} (a,b) = (a,x^{-1} \Theta (x^{-1}) \psi (xa) +
b),
$$
and observe that the map $a \mapsto x^{-1} \Theta (x^{-1}) \psi
(xa)$ is a linear transformation of $F$.  Thus,
$(\phi_\psi)^{\phi_x} \in A_1$.  Also, we have
$$
(\phi_\psi)^{\phi_\tau} (a,b) = (a,(\psi (a^\tau))^{\tau^{-1}} + b),
$$
and since the map $a \mapsto (\psi (a^\tau))^{\tau^{-1}}$ is a
linear transformation of $F$, we have $(\phi_\psi)^{\phi_\tau} \in
A_1$.  Finally,
$$
(\phi_x)^{\phi_\tau} (a,b) = (x^{\tau^{-1}}a,x^{\tau^{-1}} \Theta
(x^{\tau^{-1}}) b),
$$
and so, $(\phi_x)^{\phi_\tau} = \phi_{x^{\tau^{-1}}} \in A_2$.  We
conclude that $A_2$ and $A_3$ both normalize $A_1$ and $A_3$
normalizes $A_2$.  We note that $A_2 A_3$ is isomorphic to $\Gamma (N)$ where $A_2$ corresponds to $\Gamma_o (N)$.


We now work to prove that ${\rm Aut} (M) = A_1 A_2 A_3$.  We begin with a general lemma about the automorphisms of $M$.

\begin{lemma} \label{auts}
Let $M$ be the Suzuki $2$-group $A (n,\Theta)$ with notation as above.  If $\phi$ is an automorphism of $M$, then there exist bijective linear transformations $f$ and $h$ of $F$ that satisfies $h (a \Theta (a)) = f(a) \Theta(f(a))$ for all $a \in F$ and a map $g$ from $F$ to $F$ that satisfies $g (0) = 0$ so that $(a,b)^\phi = (f(a), g(a) + h (b))$.  If $h (\Theta (a_1) a_2) = \Theta(f(a_1)) f(a_2)$ for all $a_1, a_2 \in F$, then $g$ is a linear transformation of $F$.
\end{lemma}

\begin{proof}
Since $\phi$ is an automorphism, we have $(a,b)^\phi = (a,0)^\phi (0,b)^\phi$ for all $(a,b) \in M$.  We define $f$ and $g$ by $(a,0)^\phi = (f(a),g(a))$.  Since $\phi$ determines an automorphism of $M/N$, it follows that $f$ is a bijective linear transformation of $F$.   It is not difficult to see that $g$ is a map from $F$ to $F$.  Since $\phi$ maps $(0,0)$ to $(0,0)$, we conclude that $g (0) = 0$.  Notice that $N = Z(M)$ is characteristic, so $\phi$ must map $N$ to $N$.  It follows that we define $h$ by $(0,b)^\phi = (0,h(b))$.  Since the restriction of $\phi$ to $N$ will be an automorphism of $N$, we conclude that $h$ is a bijective linear transformation of $F$. We have $(a,0)^\phi (0,b)^\phi = (f(a),g(a))(0,h(b)) = (f(a), g(a)+h(b))$.

Observe that $(a,0)^2 = (0,\Theta (a) a)$.  Applying $\phi$ to this equation yields $((a,0)^\phi)^2 = (f(a),g(a))^2 = (0,\Theta (f(a)) f(a))$ and $(0,\Theta (a) a)^\phi = (0, h (\Theta (a) a)$.  We obtain $(0,\Theta (f(a)) f(a)) = (0,h(\Theta (a) a))$, and we have the desired equality $h (\Theta (a) a) = \Theta (f (a)) f(a)$.

We know that $(a_1,0)(a_2,0) = (a_1 + a_2,  \Theta (a_1) a_2)$.  Applying $\phi$, we obtain $(a_1,0)^\phi (a_2,0)^\phi = (f(a_1),g(a_1)) (f(a_2),g(a_2))) = (f(a_1)+f(a_2), g(a_1) +  g(a_2) + \Theta (f (a_1)) f(a_2))$ and $(a_1 + a_2,  \Theta (a_1) a_2)^\phi = (f(a_1 + a_2), h(\Theta (a_1) a_2) + g (a_1 + a_2))$.  Since $\phi$ is an automorphism, we know that $g(a_1) +  g(a_2) + \Theta (f (a_1)) f(a_2)) = h(\Theta (a_1) a_2) + g (a_1 + a_2)$.  If $\Theta (f (a_1)) f(a_2) = h(\Theta (a_1) a_2)$, then $g (a_1 + a_2) = g (a_1) + g (a_2)$, and we conclude that $g$ is a linear transformation.
\end{proof}

This next fact is Theorem VIII.6.9(b) of \cite{HBII}.

\begin{lemma}\label{bijection}
Let $F$ be a field of order $2^n$ and let $\Theta$ be an automorphism of $F$ of order $l$ where $l$ is an odd integer.  Then the map $a \mapsto a \Theta (a)$ is a bijection of $F$.
\end{lemma}


We now consider the subgroup $A_1$ in ${\rm Aut} (M)$.

\begin{lemma} \label{cent N}
Let $M$ be the Suzuki $2$-group $A (n,\Theta)$ with notation as above.  Then $C_{{\rm Aut} (M)} (N) = A_1$.
\end{lemma}

\begin{proof}
It is obvious that $A_1 \le C_{{\rm Aut} (M)} (N)$.  Suppose that $\phi \in C_{{\rm Aut} (M)} (N)$.  We can write $(a,b)^\phi = (f(a),g(a)+ h(b))$ as in Lemma \ref{auts}.  Since $\phi$ centralizes $N$, we see that $h (b) = b$ for all $b \in F$.  We know that $f(a) \Theta (f(a)) = h (a \Theta (a)) = a \Theta (a)$ for all $a \in F$.  By Lemma \ref{bijection}, we obtain $f(a) = a$ for all $a \in F$.  We conclude that $(a,b)^\phi = (a,b + g(a))$.  Since $h (\Theta (a_1) a_2) = \Theta (a_1) a_2 = \Theta (f (a_1)) f(a_2)$, we may use Lemma \ref{auts} to conclude that $g$ is a linear transformation of $F$, and so, $\phi = \phi_g \in A_1$.
\end{proof}

This next corollary encodes the result of \cite{autos}.

\begin{corollary} \label{Sylow 2}
Let $M$ be the Suzuki $2$-group $A (n,\Theta)$ with notation as above.  Let $\sigma$ generate the Sylow $2$-subgroup of the Galois group.  Then $\langle A_1, \phi_\sigma \rangle$ is a Sylow $2$-subgroup of ${\rm Aut} (M)$.
\end{corollary}

We now prove the desired result.  Our intuition is that this has been proved before, but we cannot find any reference where it is explicitly computed.

\begin{theorem}
Let $M$ be the Suzuki $2$-group $A (n,\Theta)$ with notation as above.  Then ${\rm Aut} (M) = A_1 A_2 A_3$.
\end{theorem}

\begin{proof}
The map $\rho$ from ${\rm Aut} (M)$ to ${\rm Aut} (N)$ defined by restriction is a homomorphism whose kernel is $C_{{\rm Aut} (M)} (N)$.  We know ${\rm Aut} (M)$ is solvable, so $\rho ({\rm Aut} (M))$ is solvable.  Let $X$ be a Hall $2$-complement of $\rho ({\rm Aut} (M))$.  We know that $X$ acts transitively on the nonidentity elements of $N$.  By Theorem VIII.3.5 of \cite{HBII}, we see that $X$ is isomorphic to a subgroup of the affine group of $N$.  In particular, we have that $|X| \le (2^n - 1)n_{2'}$.  We saw in Corollary \ref{Sylow 2} that $|\rho ({\rm Aut} (M))|_2 = n_2$, and so, $|\rho ({\rm Aut} (M))| \le (2^n - 1)n$.  We know that $A_2 A_3 \cap A_1 = 1$ and $|A_2 A_3| = (2^n - 1)n$.  It follows that $\rho ({\rm Aut} (M)) \cong A_2 A_3$, and hence, ${\rm Aut} (M) = A_1 A_2 A_3$.
\end{proof}

Before we apply this to obtain results in our situation, we make an observation.  Suppose $\sigma$ is a Galois automorphism of $F$.  It is not difficult to see that $\phi_\sigma \in A_3$ will fix $(1,0) \in M$.  In particular, $\phi_\sigma$ will fix an element of $M$ that is outside of $N$.  If $(G,N)$ is a pair satisfying the conclusion of \ref{2pairclass}, then we see that $N_G (K)$ cannot contain any element that induces an automorphism which is conjugate to an element of $A_3$.  Hence, if we wish to show that no such group $G$ can exist, then it suffices to prove that $N_G (K)$ would have to contain an element that induces an automorphism that is conjugate to an element of $A_3$.

\section{$p=2$ and $G$ is solvable: part 2} \label{spec}

We now use the observation from the end of the previous section to show that the groups in Conclusion 2 of Theorem \ref{2pairclass} do not exist.  We begin by determining more information about the elements of $A_1$ that commute with elements of $A_2$.  Since the elements of $C$ commute with $K$ and the image of $K$ is conjugate to a subgroup of $A_2$, this can be viewed as describing the image of $C$ in ${\rm Aut} (M)$.

\begin{lemma} \label{cent A1}
Let $M$ be the Suzuki $2$-group $A (n,\Theta)$ with notation as above.  Let $x \in F^x$.  If $1 \ne \phi_\psi \in C_{A_1} (\phi_x)$, then there is a positive integer $j$ so that $x \Theta (x) = x^{2^j}$.
\end{lemma}

\begin{proof}
Since $\phi_x$ and $\phi_\psi$ commute, we have that $\psi (a) = x^{-1} \Theta (x^{-1}) \psi (xa)$ for every element $a \in F$.  This implies that $\psi (xa) = x \Theta (x) \psi (a)$ for all $a \in F$.  Since $\phi_\psi \ne 1$, we know that $\psi \ne 0$.  Thus, there is an element $a \in F$ so that $\psi (a) \ne 0$.  Applying this with $x^2a$, we obtain $\psi (x^2a) = x \Theta (x) \psi (xa) = x^2 \Theta (x^2) \psi (a)$.  Inductively, we obtain $\psi (x^ia) = x^i \Theta (x^i) \psi (a)$ for every positive integer $i$.

Let $f (y)$ be the minimal polynomial in $Z_2 [y]$ for $x$.  We write $f (y) = \sum_{i=0}^r a_i y^i$ where $a_i \in Z_2$.  Notice that $f (y)$ is fixed by the Galois automorphisms of $F$. Thus, the roots of $f (y)$ have the form $x, x^2, x^4, \dots, x^{2^n}$.  We have
$$
0 = \psi (0) = \psi (f(x)a) = \psi (\sum_{i=1}^r a_i x^i a) = \sum_{i=1}^r a_i \psi (x^i a) $$
$$= \sum_{i=1}^r a_i (x\Theta (x))^i \psi (a) = f(x \Theta (x)) \psi (a).
$$
Since $\psi (a) \ne 0$, we conclude that $x \Theta (x)$ is a root of $f (y)$, and thus, $x \Theta (x) = x^{2^j}$ for some positive integer $j$.
%
\end{proof}

We refine the condition of the last lemma to a modular congruence.

\begin{corollary}\label{numcond}
Let $M$ be the Suzuki $2$-group $A (n,\Theta)$ with notation as above.  Let $x \in F^x$.  Suppose $\Theta (a) = a^{2^h}$ for some positive integer $h$.  If $2^h + 1 \not\equiv 2^j ~({\rm mod}~o(x))$ for all positive integers $j$, then $C_{A_1} (\phi_x) = 1$.
\end{corollary}

\begin{proof}
Suppose $C_{A_1} (\phi_x) > 1$.  By Lemma \ref{cent A1}, we know that $x \Theta (x) = x x^{2^h} = x^{2^h + 1} = x^{2^j}$ for some positive integer $j$.  It follows that $2^h + 1 \equiv 2^j ~({\rm mod}~o(x))$, a contradiction.
\end{proof}

We now come to our key observation.

\begin{lemma} \label{nontrivcent}
Let $M$ be the Suzuki $2$-group $A (n,\Theta)$ with notation as above.  Let $(G,N)$ be as in the conclusion of Theorem \ref{2pairclass}.  Let $\rho$ be the homomorphism from $N_G (K)$ to ${\rm Aut} (M)$.  If the element $g \in C_G (K)$ has odd order (if $n = 6$, then assume $g \in C_G (K)$ has order $7$), then $\rho (g)$ is conjugate to an element in $A_2$ and $C_{A_1} (\rho (g)) > 1$.
\end{lemma}

\begin{proof}
Recall that $G/{\bf O}_2 (G) \cong N_G (K)/C$ is isomorphic to a subgroup of the affine group of $N$.  If $n > 6$, then the centralizer of $K$ in the affine group is normal and cyclic of order $2^n - 1$, so $C \langle g \rangle$ is a characteristic subgroup of $C_{N_G (K)/C} (K)$.  (In particular, the centralizer of the image of $K$ in $\Gamma (N)$ is $\Gamma_o (N)$.)  If $n = 6$, it is not difficult to see that the subgroup of order $7$ is normal in the affine group.  This implies that $C \langle g \rangle$ is normal in $N_G (K)$.  We know that $\langle g \rangle$ is a Hall subgroup of $C \langle g \rangle$.  By the Frattini argument, we know that $N_G (K)  = C N_{N_G (K)} (\langle g \rangle )$.  In particular, $N_{N_G (K)} (\langle g \rangle)$ has a nontrivial Sylow $2$-subgroup.

Observe that $\rho (\langle K, g \rangle)$ has odd order, so it is conjugate to a subgroup of $A_2 A_3$.  Since $A_2A_3$ is isomorphic to the affine group of $N$, it follows that $C_{A_2A_3} (\rho (K))$ is conjugate to $A_2$, and we conclude that $\rho (g)$ is conjugate to some element of $A_2$.

Let $x \in F$ be the element so that $\rho (g)$ is conjugate to $\phi_x$.  Let $\sigma$ be the generator for the Sylow $2$-subgroup of the Galois group of $F$.  Let $P = A_1 \langle \phi_\sigma \rangle$, and we know that $P$ is a Sylow $2$-subgroup of ${\rm Aut} (M)$.  Since $\sigma$ will normalize $\langle x \rangle$, it follows that $\phi_\sigma$ will normalize $\langle \phi_x \rangle$.  Hence, $\phi_\sigma$ will normalize $A_1 \langle \phi_x \rangle$.  Let $A = P \langle \phi_x \rangle$.  We conclude that $N_A (\langle \phi_x \rangle) = C_{A_1} (\phi_x) \langle \phi_x, \phi_\sigma \rangle$.

Suppose $C_{A_1} (\rho (g)) = 1$, then $C_{A_1} (\phi_x) = 1$.  This implies that $N_A (\langle \phi_x \rangle) = \langle \phi_x, \phi_\sigma \rangle$.  It follows that a Sylow $2$-subgroup of $N_{{\rm Aut} (M)} (\langle \rho (g) \rangle)$ is conjugate to a subgroup of $\langle \phi_\sigma \rangle$.  We next observe that $N_{N_G (K)} (\langle g \rangle)$ is isomorphic to a subgroup of $N_{{\rm Aut} (M)} (\langle \rho (g) \rangle)$.  Since $N_{N_G (K)} (\langle g \rangle)$ contains a nontrivial Sylow $2$-subgroup, it contains an element that $\rho$ maps to a nontrivial power of $\phi_\sigma$.  However, $\phi_\sigma$ centralizes elements of $M$ outside of $N$, and thus, $N_G (K)$ contains an element that centralizes elements of $M$ outside of $N$.  This is a contradiction of Theorem \ref{2pairclass}.
\end{proof}

\begin{corollary} \label{nonex}
Let $M$ be the Suzuki $2$-group $A (n,\Theta)$ with notation as above.  Let $(G,N)$ be as in the conclusion of Theorem \ref{2pairclass}.  Suppose $\Theta$ has order $l$, $n = kl$, and $\Theta (a) = a^{2^h}$ for all $a \in F$.  
If $d$ divides $n$ and does not divide $k$, then $C_G (K)$ does not contain a subgroup of order $2^d - 1$.
\end{corollary}

\begin{proof}

Suppose now that $d$ divides $n$ and does not divide $k$.  (We note that when $n = 6$, it suffices to assume that $d = 3$ since we must have $l = 3$ and $k = 2$ implies that $d = 3$ or $d = 6$.  However, if we can prove that $C_G (K)$ has no element of order $2^3 - 1 = 7$, then it will contain no element of order $2^ - 1 = 63$.)  We have ${\rm gcd} (h,n) = k$, so $d$ does not divide $h$. Observe that $d$ is the order of $2$ modulo $2^d - 1$.  It follows that $2^h \cong 2^{h'} ~({\rm mod}~2^d - 1)$ where $h'$ is the remainder upon dividing $h$ by $d$, and so $1 \le h' < d$.  Now, $2^{h'} + 1$ will not be congruent to $2^j$ modulo $2^d - 1$ for all integers $j$ with $0 \le j < d$ since both $2^{h'} + 1$ and $2^j$ are less than $2^d - 1$ and not equal.  We conclude that $2^h + 1$ is not congruent modulo $2^d - 1$ to any power of $2$.

Suppose $C_G (K)$ contains an element $g$ of order $2^d - 1$. By Lemma \ref{nontrivcent}, $\rho (g)$ is conjugate to $\phi_x$ for some element $x \in F$ and $C_{A_1} (\rho (g)) > 1$.  (If $n = 6$, then assumption $d = 3$ implies that $g$ has order $7$, so the hypotheses of Lemma \ref{nontrivcent} are met in this case.)  Since $x$ must have order $2^d - 1$, we have by Corollary \ref{numcond} that $C_{A_1} (\phi_x) = 1$.  Since $C_{A_1} (\rho (g))$ and $C_{A_1} (\phi_x)$ are congruent, this is a contradiction.
\end{proof}

Combining Lemma \ref{prime pow} with Theorem \ref{2pairclass} and Corollary \ref{nonex}, we obtain the conclusion.  This is Theorem 2 when $p = 2$ and $G$ is solvable, and thus, it completes the proof of Theorem 2.

\begin{corollary}
Let $(G,N)$ be a $2$-Gagola pair with $G$ solvable.  Then $|G:N|_2 \ge |N|^2$.
\end{corollary}

\begin{proof}
If we do not have the conclusion, then $(G,N)$ satisfies conclusion 2 of Theorem \ref{2pairclass}.  We use the notation from there.  Let $l$ be the order of $\Theta$.  Let $p$ be a prime divisor of $l$, and take $d = n_p$.  Thus, $d$ divides $n$, but $d$ will not divide $n/l$.  We apply Corollary \ref{nonex} to see that $C_G (K)$ contains no element of order $2^d - 1$.  On the other hand, we know by Lemma \ref{aff} that $G/{\bf O}_2 (G)$ is isomorphic to a subgroup of $\Gamma (N)$. Let $\rho$ be the map from $G$ to $\Gamma (N)$.  It is not difficult to see $\Gamma_o (N) \le C_{\Gamma (N)} (\rho (K))$, so $\rho (H) \cap \Gamma_o (N) \le C_{\rho (H)} (\rho (K)) = \rho (C_H (K))$.  By Lemma \ref{prime pow}, we know that $\rho (H) \cap \Gamma_o (N)$ contains an element of order $2^d - 1$.  This implies that $C_H (K) {\bf O}_2 (G)/{\bf O}_2 (G) \cong C_H (K)$ contains an element of order $2^d - 1$, and so, $C_G (K)$ contains an element of order $2^d - 1$.  We now have a contradiction.
\end{proof}

\end{document}